\documentclass[10pt,draftnote,onecolumn,Romantable]{article}
\usepackage[english]{babel}
\usepackage{fullpage,lineno}
\usepackage{cite,xspace,graphicx,amsmath,amssymb} 
\usepackage{amsthm}
\usepackage[monochrome]{color}
\usepackage[colorlinks=true]{hyperref}
\usepackage{multirow}
\usepackage[font=small,skip=0pt]{caption} 
\numberwithin{equation}{section}

\theoremstyle{definition}

\theoremstyle{theorem}
\newtheorem{theorem}{Theorem}[section]
 
\usepackage{color}
   
\newcommand{\tcr}[1]{{\color{red} #1}}   

\begin{document}

\title{Mathematical Model and Optimal Control of the Transmission Dynamics of Avian Spirochaetosis (TICK Fever)}

\author{Joy I. Uwakwe$^{1}$, Simeon C. Inyama$^{2}$, Blessing O. Emerenini$^{3}$\\
$^{1}$ Department of Mathematics, Alvan Ikokwu Fed. Col. of Edu., Owerri, Imo State, Nigeria\\
$^{2}$ Department of Mathematics, Fed. Uni. of Technology Owerri, Imo State, Nigeria\\
$^{3}$ Postdoc fellow (2016-2017) Department of Physics, Ryerson University, Toronto, Ontario, Canada 
}
\date{}

\maketitle
\linenumbers

\begin{abstract}
Avian Spirochaetosis is an acute endemic tick-borne disease of birds, caused by Borrelia anserins, a species of Borrelia bacteria. In this paper, we present a compartmental Mathematical model of the disease for the bird population and Tick population. The model so constructed was analyzed using methods from dynamical systems theory. \tcr{The disease steady (equilibrium) state was determined and the conditions for the disease-free steady state to be stable were determined}.  The analysis showed that the disease-free steady state is locally stable if $d\geq \tau_B$ and $\delta \geq \tau_T$, that is, the natural death rate of birds (d) will be greater than the per capita birth rate of birds $\tau_B$ and the death rate of tick $\delta)$ is greater than  the per capita birth rate of tick $\tau_T$. This means that for the disease to be under control and eradicated within a while from its outbreak, the natural death rate of birds $d$ will be greater than the per capita birth rate of bird $\tau_B$ and the death rate of tick $\delta$ is greater than the per capita birth rate of tick $\tau_T$. It was also proved the disease-free equilibrium (DFE) and the endemic equilibrium (EE) are globally stable using Lyaponov method. Three control measures were introduced into the model. The optimality system of the three controls is characterized using optimal control theory and the existence and uniqueness of the optimal control are established. Then, the effect of the incorporation of the three controls is investigated by performing numerical simulation.

%Three control measures were applied on the model and it was shown that the three were optimal controls and existed. The control $u_1(t)$ represents the effort to reduce the number of latently infected birds such as through bio-security measures that involves regular dispose of bird’s faeces and general poultry sanitation. The control variable $u_2(t)$ represents the use of antibiotics such as penicillin to minimize the number of infectious birds. The third control variable $u_3(t)$ represents the level of insecticide such as Permerthrin used for tick control, administered at tick breeding sites to eliminate specific breeding areas. 
\end{abstract}


\section{Introduction}
Avian spirochaetosis is an acute endemic tick-borne disease of birds, caused by Borrelia anserins, a species of Borrelia bacteria. It affects a variety of avian species including chickens, turkeys, ducks, geese as well as game birds. Ticks are considered as the most important vector of disease-causing pathogens in domestic and wild animals. In many countries, avian spirochaetosis has been reported to be one of the most severe diseases affecting poultry industry. In addition to the historical importance of avian spirochaetosis, the pathogenic agent is prevalent worldwide \cite{Raquel,Ataliba}.

The parasite is spread by a soft tick (Argas persicus) of the several species of Argas ticks (fowl Tick family). The spirochete may be found in the blood of infected birds during the beginning stages of the disease process. The ticks hide during the day in cracks and crevices, suck the fowl’s blood at night and introduce the fever producing parasite Borrelia anserins. Ticks inoculate spirochetes by excretion of coxal fluid or by saliva when feeding on the birds \cite{Kaikabo}. The Ticks transmit the infection transovarially and through non-viraemic transmission \cite{Merck}. Birds transmit the disease amongst themselves through infected faeces or by contact with infected equipment \cite{Fox}. Outbreak of the disease tends to occur during the peak tick activity, during warm, humid conditions. Clinically, the disease is expressed by drowsiness, anorxia, inappetence, greenish diarrhea, hyperthermia, paralysis of the legs and wings as well as sudden death of birds. Several antibiotics agents like penicillin, tetracycline and tylosin have been seen to be very effective in treating the infected birds \cite{Borrelin}. Birds normally have protective immunity after recovering from natural infection.
%%%%%%%%%%%%%%%%%%%%%%%%%%%%%%%%%
While many experimental and field study of infectious disease spread and transmission, there is still great need for more insight into the epidemiology of infectious disease, and design of control strategies. 

Mathematical modeling has become an important tool which can be used to guide the identification of critical intervention points aimed at minimizing disease related mortality. Several mathematical studies have been made in the area of tick-borne disease with findings on potential management strategies \cite{Porter}, control effort for treatment of host and prevention of host-vector contact with minimal cost and side effect \cite{Hee-Dae}. Other related models have placed more emphasis on non-viraemic transmission \cite{Rosa}, relationship between vectors and their host and its correlation to tick-borne encephalitis infections in the region. Majority of the models are based on the deterministic SIR type model consisting of coupled ordinary differential equations \cite{Rosa, Gilbert}, SEIR models \cite{Liuyong}. Other models have used optimal control theory to obtain an optimal vaccination strategy using critical threshold values of vaccine coverage ratio.

In this study, we develop a mathematical modeling framework that incorporates differences in a given population based on susceptibility, exposure and recovery. Specifically, we apply optimal control to the transmission dynamics of avian spirchaetosis disease in poultry birds. In addition to the model formulation, we address the question of existence of steady states and stability of disease free equilibrium through the mathematical analysis and numerical simulations. Our goal is to determine optimal strategy model for the prevention and treatment of avian spirochoetosis in order to reduce incidence rate in poultry.

%%%%%%%%%%%%%%%%%%%%%%%%%%%%%%%%%

%\section{Formulation of Avian Spirochaetosis (Tick Fever) Model}
\section{Method}
\subsection{Basic Model Assumptions}
We formulate a mathematical model that describes the transmission dynamics of the Avian Spirochaetosis within a poultry population. We assume there exists transovarial transmission among ticks (transmission from adult female tick to egg/larvae); there is non-viraemic transmission amongst the tick (that is, susceptible ticks can be infected through co-feeding with an infected tick). We also assume that recovered birds develop permanent immunity to the disease and there is no recovery for infected ticks since ticks have a short life span.

\subsection{Governing equations: Avian Spirochaetosis Model and the Control Model}
Using the assumptions, the model describing the transmission dynamics of Avian Spirochaetosis as described in the compartmental diagram [\ref{flowdiagram}] is formulated as a system of ordinary differential equations (ODEs). The dependent variables are $S_B, E_B, I_B, S_T, E_T, I_T,R$ which at time $t$ represents the susceptible bird population, exposed bird population, infectious bird population, susceptible tick population, exposed tick population, infectious tick population and recovered birds respectively. We denote the total bird population by $N_B$ and total tick population by $N_T$ such that $N_B = S_B+E_B+I_B+R$ and $N_T = S_T+E_T+I_T$ respectively. The governing equations for the seven compartments are presented in (\ref{avian1}).

We also present in (\ref{control1}) the optimal control strategies for the disease model with transovarial and non-viraemic transmission. The system (\ref{avian1}) is modified by the inclusion of three control variables which are introduced in effort to reduce (i) the number of latently infected birds (through bio-security measures), (ii) the cost of treatment of infected birds and (iii) cost of eliminating the tick. The description, values and sources of the models parameters are summarized in table \ref{tablea}.

\begin{table}[ht]
%\caption{The Model and the control}
\centering
\begin{tabular}{p{8cm}|p{9cm}}
 \hline\\
\textbf{Avian Spirochaetosis Model} & \textbf{Control Model} \\
\begin{equation}
 \begin{array}{l}
\frac{dS_B}{dt} = \tau_BN_B - \beta_1I_TS_B - \beta_2I_BS_B - dS_B\,\\
\\
\frac{dE_B}{dt} = \beta_1I_TS_B + \beta_2I_BS_B - \alpha_BE_B -dE_B\,\\
\\
\frac{dI_B}{dt} = \alpha_BE_B - \sigma I_B -dI_b- \mu I_B\,\\
\\
\frac{dR}{dt} = \sigma I_B - dR\,\\
\\
\frac{dS_T}{dt} = \tau_TN_T - \beta_3I_BS_T - \theta I_TS_T - \lambda S_TI_T - \delta S_T\,\\
\\
\frac{dE_T}{dt} = \beta_3I_BS_T + \theta I_TS_T + \lambda S_TI_T - \delta E_T + \alpha_TE_T\,\\
\\
\frac{dI_T}{dt} = \alpha_TE_T - \delta I_T\,
\end{array} \label{avian1}
\end{equation}
&
\begin{equation}
\begin{array}{l}
\frac{dS_B}{dt} = \tau_BN_B - (1-u_1)\beta_1I_TS_B - (1-u_2)\beta_2I_BS_B - dS_B\,\\
\\
\frac{dE_B}{dt} = (1-u_1)\beta_1I_TS_B + (1-u_2)\beta_2I_BS_B - \alpha_BE_B -dE_B\,\\
\\
\frac{dI_B}{dt} = \alpha_BE_B - \alpha_BI_B - \mu I_B\,\\
\\
\frac{dR}{dt} = u_2 I_B - dR\,\\
\\
\frac{dS_T}{dt} = \tau_TN_T(1-u_3) - \beta_3I_BS_T - \theta I_TS_T - \lambda S_TI_T - \delta S_T\,\\
\\
\frac{dE_T}{dt} = \beta_3I_BS_T + \theta I_TS_T + \lambda S_TI_T - \delta E_T + \alpha_TE_T\,\\
\\
\frac{dI_T}{dt} = \alpha_TE_T - \delta I_T\,
\end{array}\label{control1}
\end{equation}\\
with initial conditions:
\begin{equation}
\left.\begin{array}{l}
S_B(0) = S_{B0}\,
\\
E_B(0) = E_{B0}\, 
\\
I_B(0) = I_{B0}\, 
\\
R(0) = R_{0}\, 
\\
S_T(0) = S_{T0}\, 
\\
I_T(0) = I_{T0}
\end{array}\right\}\label{initialcond}
\end{equation}
%S_B(0) = S_{B0}, E_B(0) = E_{B0}, I_B(0) = I_{B0}, R(0) = R_{0}, S_T(0) = S_{T0}, I_T(0) = I_{T0}
%\end{align}
\end{tabular}
\end{table}
%
%\begin{equation}
%S_B(0) = S_{B0}, E_B(0) = E_{B0}, I_B(0) = I_{B0}, R(0) = R_{0}, S_T(0) = S_{T0}, I_T(0) = I_{T0} \label{initalcond}
%\end{equation}

\begin{figure}[h!]
\begin{center}
\resizebox{0.65\linewidth}{!}{\includegraphics{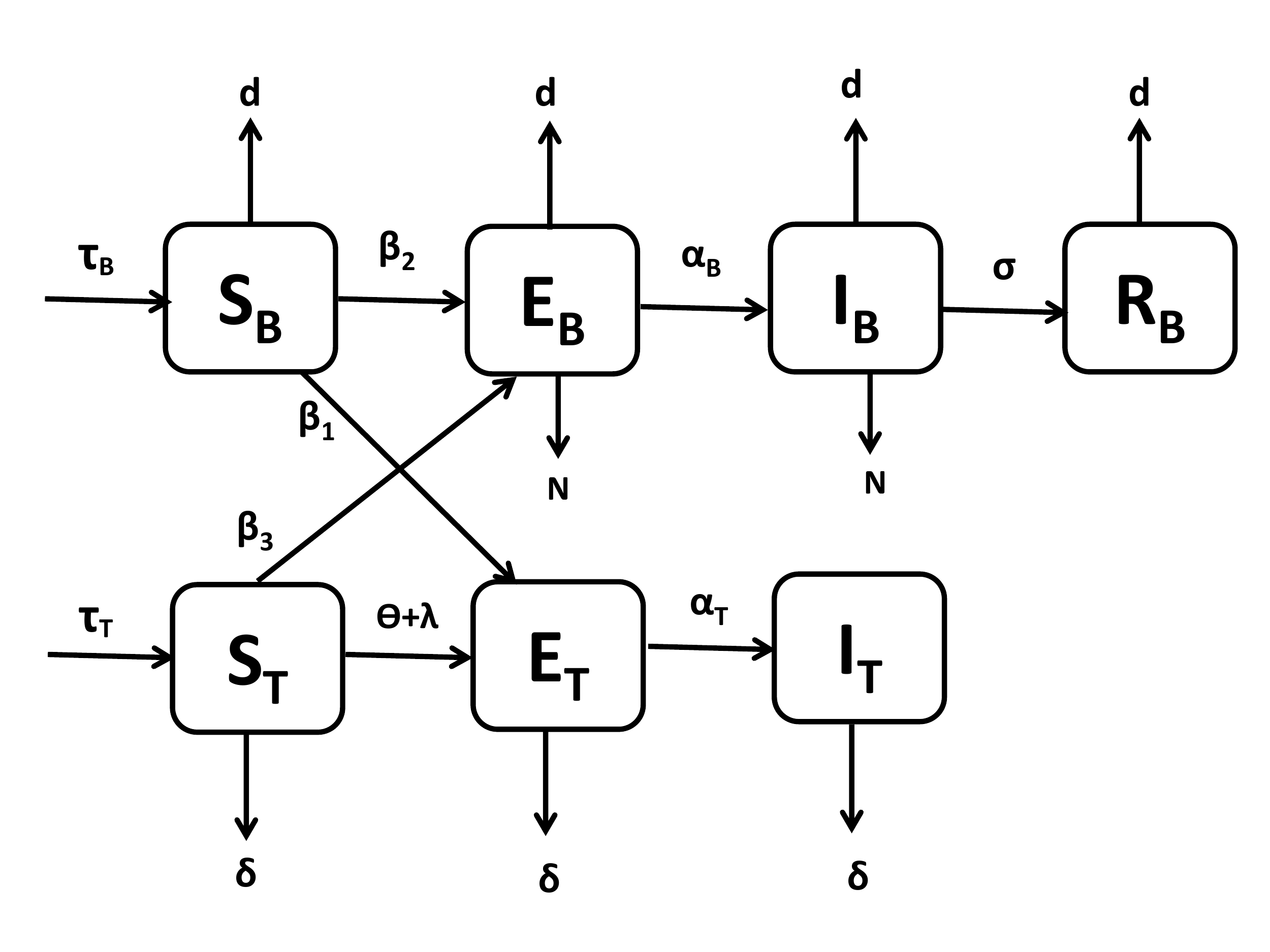}}\hspace{4mm}
\end{center}
\caption{Flow diagram for the ODE model with classes $S_B, E_B, I_B, S_T, E_T, I_T,R$. The solid lines denote transitions between classes and death rates in the model. }\label{flowdiagram}
\end{figure}

%%%%%%%%%%%%%%%%%%%%%%%%%%%%%%%%%%%%
\begin{table*}[h!]
\begin{center}
\caption{Parameters for simulation presented in Figures 1,2,3}
\begin{tabular}{ l l l l }\hline\hline
%\toprule
Parameter & Description & Value & Source  \\\hline\hline 
	$N_B$	&	Size of total bird population & $50$ &    \\
	$N_T$	&	Size of total tick population & $100$  &   \\
	$S_B(0)$	&	Susceptible bird population at time t & $100$  & Assumed  \\
	$S_T(0)$	&	Susceptible tick population at time t & $100$  & Assumed   \\
	$E_B(0)$	&	Exposed bird population at time t & $80$ & Assumed  \\
	$E_T(0)$	&	Exposed tick population at time t	 & $80$ & Assumed  \\
	$I_B(0)$	&	Infectious bird population at time t & $80$  & Assumed  \\
	$I_T(0)$	&	Infectious tick population at time t & $80$  & Assumed  \\
	$R(0)$	&	Recovered birds at time t & $60$ & Assumed  \\
	$\lambda$	&   The rate an infected adult female tick reproduces & $3.68*(10^-4)$ &   \\
	$\beta_1$	&	The rate at which a tick bites and infects a bird  &$2*(10^-4)$  &   \\
	$\beta_2$	&	The rate at which birds are infected through ingested faeces  & $0.05$ &  \\
	$\beta_3$	&	The rate a tick bites a bird and become infected  & $1.95*(10^-3)$ &  \\
	$\theta$	&	The rate of non – viraemic transmission between co – feeding ticks  & $3.9*(10^-7)$  &  \\
	$\alpha_B$	&	rate of progression from exposed to infectious class among the bird  & $0.182$   &  \\
    $\alpha_T$  &   rate of progression from exposed to infectious class among the tick &$0.182$   &  \\
	$d$	&	Natural death rate of birds & $0.087$ &   \\
	$\mu$	&	Disease induced death rate of birds & $0.2$ &   \\
	$\sigma$	&	Rate of recovery for birds & $1.25$  &  \\
	$\delta$	&	Death rate of tick  & $0.083$  &  \\
	$u_1$	&	effort to reduce the number of exposed birds & $0.02$  &   \\
    $u_2$	&	measures the rate of treatment of the infected birds & $0.01$  &   \\
	$u_3$	&	effective tick control measure & $0.05$  &   \\
   $\tau_B$     &   Per capita birth rate of bird  & $8.33$ &   \\
   $\tau_T$     &   Per capita birth rate of tick & $0.167$  &   \\\hline\hline 
\end{tabular}\label{tablea}
\end{center}
\end{table*}

%%%%%%%%%%%%%%%%%%%%%%%%%%%%%%%%%%%%%%%%%%%%%%%%%%%%%%%%%%%%%%%%%%%%%%
%%%%%%%%%%%%%%%%%%%%%%%%%%%%%%%%
\section{Analysis of the Model}
\begin{theorem}\label{t3.1}
All feasible solutions of the model (\ref{avian1}) are uniformly bounded in a proper subset.

$\phi = \phi_B \times \phi_T$ where $\phi_B = \{(S_B,E_B,I_B,R):N_B\leq  \frac{\tau_B}{d}\}$ and $\phi_T = \{(S_T,E_T,I_T):N_T\leq \frac{\tau_{T}}{\delta}\}$
\end{theorem}

\begin{proof}
We assume the associated parameters of the model (\ref{avian1}) are non – negative for all time $t>0$.
To show that all feasible solutions are uniformly bounded in a proper subset, we consider the bird and tick populations respectively i.e. $N_B= S_B + E_B + I_B + R$ and $N_T = S_T + E_T + I_T$.
Let $(S_B,E_B,I_B,R) \in R^4_+$ and $(S_T, E_T, I_T)\in R^3_+$ be any solution with non-negative initial conditions.
By differential inequality, it follows that,
\begin{equation}
\begin{split}
 \lim_{t\rightarrow\infty} \sup S_B(t) & \leq \frac{\tau_B}{d}\\
\lim_{t\rightarrow\infty} \sup S_T(t) & \leq \frac{\tau_B}{\delta}
\end{split}
\end{equation}

where $\tau_B=\tau_BN_B$ and $\tau_T=\tau_TN_T$

Taking the time derivative of $N_B$ and $N_T$ along a solution path of the model (\ref{avian1}) gives,
\begin{equation}
\begin{split}
\frac{dN_B}{dt}& = \tau_B - dN_B - \mu_BI_B\\
\frac{dN_T}{dt} & = \tau_T - \delta N_T\\
\end{split}
\end{equation}

Then,

\begin{equation}
\begin{split}
\frac{dN_B}{dt}& \leq \tau_B - dN_B\\
\frac{dN_T}{dt} & \leq \tau_T - \delta N_T\\
\end{split}
\end{equation}

and
\begin{equation}
\begin{split}
0\leq N_B & \leq \frac{\tau_B}{d} + dN_B(0)e^{-d t}\\
0\leq N_T & \leq \frac{\tau_T}{\delta} - \delta N_T(0)e^{-\delta t}\\
\end{split}
\end{equation}

where $N_B(0)$ and $N_T(0)$ are the initial values of the respective variables in each population. Thus as $t\rightarrow\infty$, then
$$0\leq N_B\leq\frac{\tau_B}{d} \ \text{and} \ 0\leq N_T\leq\frac{\tau_T}{\delta}$$

Hence these shows that $N_B$ and $N_T$ are bounded and all feasible solutions of $S_B$, $E_B$, $I_B$ ,$R$, $S_T$, $E_T$ and $I_T$ starting in the regions $\phi_B$  and $\phi_T$ will either approach, enter or stay in the region where 

$$\phi_B = \{(S_B,E_B,I_B,R):N_B\leq\frac{\tau_B}{d}\} \ \text{and} \ \phi_B = \{(S_T,E_T,I_T):N_T\leq\frac{\tau_T}{\delta}\}$$ 

\end{proof}

Therefore $N_B$ and $N_T$ are bounded and all the possible solutions of the model (\ref{avian1}) approach or stay in region $\phi=\phi_B\times\phi_B \forall t \geq 0$. Thus $\phi$ is positively invariant and the existence, uniqueness and continuity results also hold for the model (\ref{avian1}) in $\phi$. Hence the model is well-posed mathematically and epidemiologically.  

%%%%%%%%%%%%%%%%%%%%%%%%%%%%%%%%%%%%%%%%%%%%%%%%%%%%%%%%%%%%%%%%%%%%%%%%
\subsection{Existence of steady states}
The system is in a steady state if, $\frac{d S_B}{dt}=\frac{d E_B}{dt}=\frac{d I_B}{dt}=\frac{d R}{dt}=\frac{d S_T}{dt}=\frac{d E_T}{dt}=\frac{d I_T}{dt}= 0$, that is, 

\begin{equation}
\begin{split}
\tau_BN_B - \beta_1I_TS_B - \beta_2I_BS_B - dS_B &=0 \\
\beta_1I_TS_B - \beta_2I_BS_B - \alpha_BE_B - dE_B &=0  \\
\alpha_BE_B  - \sigma I_B - dI_B - \mu I_B &=0  \\
\sigma I_B - dR &=0  \\
\tau_TN_T - \beta_3I_BS_T - \lambda S_TI_T - \delta S_T &=0  \\
\beta_3I_BS_T - \theta I_TS_T - \delta E_T - \alpha_TE_T &=0  \\
\alpha_TE_T - \delta I_T &=0 
\end{split}\label{steadystate}
\end{equation}

Solving equations (\ref{steadystate}) for $S^0_B,E^0_B,I^0_T,R^0_B,S^0_T,E^0_T,I^0_T$ we have the steady (equilibrium) state as follows:

\begin{equation}
\begin{split}
S_B^0 & = \frac{{{\tau _B}N_B^0\,}}{{({\beta _1}I_T^0 + {\beta _2}I_B^0 + d)}}\,\\
E_B^0 & = \frac{{(\sigma  + d + \mu )I_B^0}}{{{\alpha _B}}}\,\\
\,I_B^0 & = \frac{{{\alpha _B}{\beta _1}I_T^0S_B^0}}{{\left[ {({\alpha _B} + d)(\sigma  + d + \mu ) - {\alpha _B}{\beta _2}} \right]}}\,\,\\
R^0 & = \frac{{\sigma I_B^0}}{d}\\
S_T^0 & = \frac{{{\tau _T}N_T^0}}{{({\beta _3}I_B^0 + \theta I_T^0 + \lambda I_T^0 + \delta )}}\\
E_T^0 & = \frac{{{\tau _T}N_T^0({\beta _3}I_B^0 + \theta I_T^0 + \lambda I_T^0)}}{{({\beta _3}I_B^0 + \theta I_T^0 + \lambda I_T^0 + \delta )(\delta  - {\alpha _T})}}\\
\,I_T^0 & = \frac{{{\alpha _T}{E_T}}}{\delta }\,
\end{split}
\end{equation}

The disease free steady (equilibrium) state for the disease is ${E_0} = (\frac{{{\tau _B}\mathop N\nolimits_B }}{d},0,0,0,\frac{{{\tau _T}\mathop N\nolimits_T }}{\delta },0,0)$

%%%%%%%%%%%%%%%%%%%%%%%%%%%%%%%%%%%%%%%%%%%%%%%%%%%%%%%%%%%%%%%%%%%%%%%%%%%%%%%%%%%%%%

\subsection{Local Stability of Disease-free (DFE) Steady (equilibrium) state}
Linearizing the system (\ref{avian1}) we have the Jacobian matrix as 
\begin{equation}
{J_{DF}} = \left( {\begin{array}{*{20}{c}}
a&{{\tau _B}}&{{\tau _B} - {\beta _2}S_B^0}&{{\tau _B}}&0&0&{ - {\beta _1}S_B^0}\\
b&{ - {\alpha _B} - d}&{{\beta _2}S_B^0}&0&0&0&{{\beta _1}{S_B}}\\
0&{{\alpha _B}}&{ - (\sigma  + d + \mu )}&0&0&0&0\\
0&0&\sigma &{ - d}&0&0&0\\
0&0&{ - {\beta _3}S_T^0}&0&c&0&e\\
0&0&{{\beta _3}S_T^0}&0&f&{ - (\delta  + {\alpha _T})}&{\theta S_T^0 + \lambda S_T^0}\\
0&0&0&0&0&{{\alpha _T}}&{ - \delta }
\end{array}} \right)
\end{equation}
where
$a = {\tau _B} - {\beta _1}I_T^0 - {\beta _2}I_B^0 - d$, \ $b = {\beta _1}I_T^0 + {\beta _2}I_B^0$, $c = {\tau _T} - {\beta _3}I_B^0 - \theta I_T^0 - \lambda I_T^0 - \delta $, $e = {\tau _T} - \theta S_T^0 - \lambda S_T^0$ \ and \ 
$f = {\beta _3}I_B^0 + \theta I_T^0 + \lambda I_T^0$
%%%%%%%%%%%%%%%%%%%%%%%%%%%%%%%%%%%%%%%%%%%%%%%%%%%%%%%%%%%%%%%%%%%%%%%%%%%%%%%%%%

\begin{theorem}
The disease-free steady (equilibrium) state of the model (\ref{avian1}) is locally asymptotically stable if $d \ge {\tau _B}\,\,\,and\,\,\,\,\delta  \ge {\tau _T}$.
\end{theorem}

\begin{proof}
Evaluating $J_{DF}$   at $E_0$ we have, 

\begin{equation}
{J_{D{F_0}}} = \left( {\begin{array}{*{20}{c}}
a&{{\tau _B}}&{{\tau _B} - {\beta _2}\frac{{{\tau _B}}}{d}}&{{\tau _B}}&0&0&{\frac{{ - {\beta _1}{\tau _B}}}{d}}\\
b&{ - {\alpha _B} - d}&{{\beta _2}\frac{{{\tau _B}}}{d}}&0&0&0&{\frac{{{\beta _1}{\tau _B}}}{d}}\\
0&{{\alpha _B}}&{ - (\sigma  + d + \mu )}&0&0&0&0\\
0&0&\sigma &{ - d}&0&0&0\\
0&0&0&0&c&0&e\\
0&0&0&0&f&{ - (\delta  + {\alpha _T})}&0\\
0&0&0&0&0&{{\alpha _T}}&{ - \delta }
\end{array}} \right)
\end{equation}

\[\left| {{J_{D{F_0}}} - \lambda I} \right| = \left| {\begin{array}{*{20}{c}}
{a - \lambda }&{{\tau _B}}&{{\tau _B} - {\beta _2}\frac{{{\tau _B}}}{d}}&{{\tau _B}}&0&0&{\frac{{ - {\beta _1}{\tau _B}}}{d}}\\
b&{ - {\alpha _B} - d - \lambda }&{{\beta _2}\frac{{{\tau _B}}}{d}}&0&0&0&{\frac{{{\beta _1}{\tau _B}}}{d}}\\
0&{{\alpha _B}}&{ - (\sigma  + d + \mu ) - \lambda }&0&0&0&0\\
0&0&\sigma &{ - d - \lambda }&0&0&0\\
0&0&0&0&{c - \lambda }&0&e\\
0&0&0&0&f&{ - (\delta  + {\alpha _T}) - \lambda }&0\\
0&0&0&0&0&{{\alpha _T}}&{ - \delta  - \lambda }
\end{array}} \right|\]
%%%%%%%%%%%%%%%%%%%%%%%%%%%%%%%%%%%%%%%%%%%%%%%%%%%
%%%%%%%%%%%%%%%%%%%%%%%%%%%%%%%%%%%%%%%%%%%%%%%%%
The eigen-value of the system are: $a,\,\, - {\alpha _B} - d,\,\, - (\sigma  + d + \mu ),\,\, - d,\,\,c,\,\, - (\delta  + {\alpha _T}),\,\, - \delta \,\,$. Hence the disease free steady state is asymptotically stable only if 
	\[a = {\tau _B} - d < 0\,\, \Rightarrow \,{\tau _B} < d\, \Rightarrow d > {\tau _B}\] and
	\[c = {\tau _T} - \delta  < 0\,\, \Rightarrow {\tau _T} < \delta \,\, \Rightarrow \delta \, > {\tau _T}\]
Thus, for the disease-free steady state to be stable, $d > {\tau _B}$ and $\delta \, > {\tau _T}$, in other words, the natural death rate of birds $d$ will be greater than the per capita birth rate of bird ${\tau _B}$ and the death rate of tick $\delta$ is greater than the per capita birth rate of tick ${\tau _T}$. 
\end{proof}
This means that for the disease to be under control and eradicated within a while from its outbreak, the natural death rate of birds $d$ will be greater than the per capita birth rate of bird ${\tau _B}$ and the death rate of tick $\delta$ is greater than the per capita birth rate of tick ${\tau _T}$.

\subsection{Local Stability of Endemic Steady (equilibrium) State}	
In this section we investigate the stability of the endemic state  $E_E^* = (0,\,E_B^0,\,I_B^0,\,0,\,0,\,E_T^0,\,I_T^0)$ where $E_B^0,\,I_B^0,\,E_T^0,\,I_T^0$ are as defined in (\ref{eq3.2}) below: Note that  $S_B^0 = \,R^0 = \,S_T^0 = 0$
\begin{equation}
\begin{split}
E_B^0 & = \frac{(\sigma  + d + \mu)I_B^0}{\alpha _B}\\
I_B^0 & = \frac{\alpha _B\beta _1 I_T^0S_B^0}{(\alpha _B + d)(\sigma  + d + \mu ) - \alpha _B\beta _2}\\
E_T^0 & = \frac{{{\tau _T}N_T^0({\beta _3}I_B^0 + \theta I_T^0 + \lambda I_T^0}}{{{\beta _3}I_B^0 + \theta I_T^0 + \lambda I_T^0 + \delta )(\delta  - {\alpha _T})}}\\
I_T^0 & = \frac{{{\alpha _T}{E_T}}}{\delta }
\end{split}\label{eq3.2}
\end{equation}

\begin{theorem}
The endemic steady (equilibrium) state is locally asymptotically stable if $E_T^0 > \frac{{\delta ({\tau _B} - d)}}{{{\beta _1}{\alpha _T}}}$ or if $E_T^0 > \left( {\frac{\delta }{{{\alpha _T}}}} \right)\left( {\frac{{\left( {{\tau _T} - \delta } \right)}}{{(\theta  + \lambda )}}} \right)$\\\\
\end{theorem}
\begin{proof}
Evaluating J at $E_E^*$ we have\\\\
$J = \left( {\begin{array}{*{20}{c}}
{a'}&{{\tau _B}}&{{\tau _B}}&{{\tau _B}}&0&0&0\\
b&{ - {\alpha _B} - d}&0&0&0&0&0\\
0&{{\alpha _B}}&{ - (\sigma  + d + \mu )}&0&0&0&0\\
0&0&\sigma &{ - d}&0&0&0\\
0&0&0&0&{c'}&0&{{\tau _T}}\\
0&0&0&0&{f'}&{ - (\delta  + {\alpha _T})}&0\\
0&0&0&0&0&{{\alpha _T}}&{ - \delta }
\end{array}} \right)$\\\\

where\,\,\,\,$a' = \frac{\delta_{\tau B}- B_1 α_T E_T- \delta d}{\delta}, b' = \frac{\beta_1\alpha_TE_T}{\delta}, c' = \frac{\delta_{\tau B}-(\theta + \lambda)\alpha_TE_T-\delta^2}{\delta}, f' = -\frac{(\theta + \lambda)\beta_1\alpha_TE_T}{\delta}$\\\\

$\left| {J - I\lambda } \right| = \\ \left| {\begin{array}{*{20}{c}}
{a' - {\lambda _1}}&{{\tau _B}}&{{\tau _B}}&{{\tau _B}}&0&0&0\\
b&{ - {\alpha _B} - d - {\lambda _2}}&0&0&0&0&0\\
0&{{\alpha _B}}&{ - (\sigma  + d + \mu ) - {\lambda _3}}&0&0&0&0\\
0&0&\sigma &{ - d - {\lambda _4}}&0&0&0\\
0&0&0&0&{c' - {\lambda _5}}&0&{{\tau _T}}\\
0&0&0&0&{f'}&{ - (\delta  + {\alpha _T}) - {\lambda _6}}&0\\
0&0&0&0&0&{{\alpha _T}}&{ - \delta  - {\lambda _7}}
\end{array}} \right| = 0$
\\\\\\
Hence, the roots of $\left| {J - I\lambda } \right| = 0$  are\\
$$a' - {\lambda _1},\, - {\alpha _B} - d - {\lambda _2}, - (\delta  + d - \mu ) - {\lambda _3}, - d - {\lambda _4},c' - {\lambda _5}, - (\delta  + {\alpha _T}) - {\lambda _6}\,\,and\,\, - \delta  - {\lambda _7}$$

This means that the eigen-values of the characteristic matrix are: \\ 
$$a',\, - ({\alpha _B} + d),\, - (\delta  + d - \mu ),\, - d,\,c',\, - (\delta  + {\alpha _T}),\, - \delta $$

%This means that the eigen-values of the characteristic matrix are:  \\
%$$a',\, - ({\alpha _B} + d),\, - (\delta  + d - \mu ),\, - d,\,c',\, - (\delta  + {\alpha _T}),\, - \delta $$

Since second, third, fourth, sixth and seventh eigen-values are negative, it means that first and fifth eigen-values will be stable if this two eigen-values are also negative. This means that the conditions for the endemic steady state to be stable are $a' < 0\,\,$ and\,\, $c' < 0$. This implies that\\

\begin{equation}
\frac{{\delta {\tau _B} - {\beta _1}{\alpha _T}{E_T} - \delta d}}{\delta } < 0\label{eqi}
\end{equation}

 and
 
\begin{equation}
\frac{{\delta {\tau _T} - (\theta  + \lambda ){\alpha _T}{E_T} - \delta d}}{\delta }\, < 0 \label{eqii}
\end{equation}

From (\ref{eqi}) we have that, 

$$
[\delta ({\tau _B} - d) - \left( {{\beta _1}{\alpha _T}E_T^0} \right)] < 0 \Rightarrow E_T^0 > \frac{{\delta ({\tau _B} - d)}}{{{\beta _1}{\alpha _T}}}
$$

From (\ref{eqii}) we have that, 

$$\delta_T-(\theta + \lambda) \alpha_T E_T- \delta^2< 0 \ \ \Rightarrow E_T^0 > \frac{{\left( {{\tau _T} - \delta } \right)\delta }}{{(\theta  + \lambda ){\alpha _T}}} = \left( {\frac{\delta }{{{\alpha _T}}}} \right)\left( {\frac{{\left( {{\tau _T} - \delta } \right)}}{{(\theta  + \lambda )}}} \right)$$
\end{proof}

This means that the endemic steady state is stable if the exposed Tick population at time $t$, $(E_T)$ is greater than

\begin{itemize}\tcr{
\item[(i)] the ratio of the product of the death rate of ticks and the difference between per capita birth rate of birds (${\tau _B}$) and the natural death rate of birds (d) and the product of the rate of infection of susceptible  bird by infected tick (${\beta _1}$) and rate of progression from exposed to infectious class among the tick (${\tau _T}$) and 
\item[(ii)] The product of the ratios of the death rate of tick ($\delta $) and rate of progression from exposed to infectious class among the tick (${\tau _T}$) and the difference between the rate of progression from exposed to infectious class among the tick (${\tau _T}$)  and the death rate of tick ($\delta $) and the sum of the rate of non-viraemic transmission between co-feeding ticks ($\theta $ ) and the rate an infected adult female tick reproduces ($\lambda $).}
\end{itemize}

This means that when the above two conditions are satisfied, the endemic steady state will be stable, that means that the disease will persist at both populations (that is birds and ticks).

\subsection{Global stability of the DFE}
Using the Next-generation matrix method, the basic reproduction number $R_0$ for avian Spirochaetosis disease is given as
$$R_0 = [\frac{1}{2}(R_B+R_T )+((R_B+〖R_T)〗^2-4(R_B R_T+R_TB ))]^{\frac{1}{2}}$$
With 
$$R_B=  \frac{\beta_2 \tau_B \alpha_B}{d(d+\alpha_B)(d+\mu+\sigma)}  , R_T=  \frac{\tau_T \alpha_T (\theta+\lambda)}{\delta^2 (\alpha_T+\delta)} \  \text{and} \tcr{\ R_{TB}=  \frac{\beta_2 \tau_B \alpha_B \beta_1 \tau_T \alpha_T}{d\delta^2 (\alpha_T+\delta)(d+\alpha_B)(d+\mu+\sigma)}}$$

The basic reproduction number $R_0$ reflects the infection transmitted from bird to bird $(R_B)$ through infected faces, tick to tick $(R_T)$  through non-viraemic and vertical transmission, tick to bird and bird to tick $(R_{TB})$ either by feeding on infected bird or biting a susceptible bird.\\\\

\begin{theorem}
The disease free steady state is globally asymptotically stable if $R_0<1$.\\
\end{theorem}

\begin{proof}
A comparison theorem will be used for the proof. Let $S_B = S_T = S$. The equations for the infected components of the model (\ref{avian1}) can be written as 
\tcr{	\[\begin{array}{l}
\frac{d E_B}{dt}  = ({\beta _1}{I_T} + {\beta _2}{I_B})S - ({\alpha _B} + d){E_B}\\
\frac{d I_B}{dt}= {\alpha _B}{E_B} - (\sigma  + d + \mu ){I_B}\\
\frac{d E_T}{dt} = ({\beta _3}{I_B}{S_T} + \theta {I_T} + \lambda {I_T})S - (\delta  + {\alpha _T}){E_T}\\
\frac{d I_T}{dt} ={\alpha _T}{E_T} - \delta {I_T}
\end{array}\]
}

These equations can be simplified as follows
\begin{equation}
\left(\begin{array}{c} 
\frac{d E_B}{dt} \\
\frac{d I_B}{dt} \\
\frac{d E_T}{dt} \\
\frac{d I_T}{dt}\\
\end{array}\right) 
= (S)F \left(\begin{array}{c} 
E_B \\
I_B \\
E_T \\
I_T\\
\end{array}\right) 
- V \left(\begin{array}{c} 
E_B \\ 
I_B \\
E_T \\
I_T\\
\end{array}\right) 
\end{equation}

\begin{equation}
= (S)F \left(\begin{array}{c} 
E_B \\ 
I_B \\
E_T \\
I_T\\
\end{array}\right) 
- F \left(\begin{array}{c} 
E_B \\ 
I_B \\
E_T \\
I_T\\
\end{array}\right)
+F \left(\begin{array}{c} 
E_B \\ 
I_B \\
E_T \\
I_T\\
\end{array}\right) 
- V \left(\begin{array}{c} 
E_B \\ 
I_B \\
E_T \\
I_T\\
\end{array}\right)
\end{equation}
\begin{equation}
=(F- V) \left(\begin{array}{c} 
E_B \\ 
I_B \\
E_T \\
I_T\\
\end{array}\right)
-(1-S)F\left(\begin{array}{c} 
E_B \\ 
I_B \\
E_T \\
I_T\\
\end{array}\right)
\leq (F- V) \left(\begin{array}{c} 
E_B \\ 
I_B \\
E_T \\
I_T\\
\end{array}\right)
\end{equation}
%\begin{equation}
%\leq (F- V) \left(\begin{array}{c} 
%E_B \\\\ 
%I_B \\\\
%E_T \\\\
%I_T\\\\
%\end{array}\right)
%\end{equation}

From the proof of the local asymptotic stability, the DFE is locally asymptotically stable when all the eigenvalues of the Jacobian matrix have negative real parts or equivalently when $ρ(FV^{-1} )<1$. This is equivalent to the statement that all eigenvalues of $F-V$ have negative real parts when $R_0< 1$. Therefore the linearized differential inequality is stable whenever $R_0< 1$. Consequently, by the comparison theorem, we have 
	$$(E_B, I_B, E_T, I_T)\rightarrow(0,0,0,0) \ \text{as} \ t\rightarrow\infty$$
	
Substituting $E_B =I_B =E_T = I_T = 0$ in the model gives 

	$$(S_B, R, S_T)\rightarrow(\frac{\tau_B}{d},0,\frac{\tau_T}{\delta}) \ \text{as} \ t\rightarrow\infty$$
Therefore, 
$$(S_B, E_B, I_B, S_T, E_T, I_T) \rightarrow(\frac{\tau_B}{d}  0,0,0,\frac{\tau_T}{\delta},0,0) \ \text{as} \ t\rightarrow\infty$$
and hence, the DFE is globally asymptotically stable whenever $R_0<1$. \\
\end{proof}

The epidemiological implication of the above result is that Avian Spirochaetosis disease can be eliminated from the population if the basic reproduction number can be brought down to and maintained at a value less than unity. Therefore, the condition $R_0< 1$ is a necessary and sufficient condition for the disease elimination.

\subsection{Global stability of the endemic equilibrium}
\begin{theorem}
The endemic steady (equilibrium) state is globally asymptotical stable if $R_0>1$
\end{theorem}

\begin{proof}
We consider the non-linear Lyaponuv function of Goh-Voltera type for the system

%\scriptsize{\[L = {S_B} - S_B^* - S_B^*\ln \frac{{{S_B}}}{{S_B^*}} + {E_B} - E_B^* - E_B^*\ln \frac{{{E_B}}}{{E_B^*}} + A({I_B} - I_B^* - I_B^*\ln \frac{{{I_B}}}{{I_B^*}}) + {S_T} - S_T^* - S_T^*\ln \frac{{{S_T}}}{{S_T^*}} + {E_T} - E_T^* - E_T^*\ln \frac{{{E_T}}}{{E_T^*}} + B({I_T} - I_T^* - I_T^*\ln \frac{{{I_T}}}{{I_T^*}})\]
%}
%\\
%\normalsize
\begin{equation}
\begin{split}
L = {S_B} - S_B^* & - S_B^*\ln \frac{{{S_B}}}{{S_B^*}} \\
& + {E_B} - E_B^* - E_B^*\ln \frac{{{E_B}}}{{E_B^*}} + A({I_B} - I_B^* - I_B^*\ln \frac{{{I_B}}}{{I_B^*}}) 
 {S_T} - S_T^* - S_T^*\ln \frac{{{S_T}}}{{S_T^*}} \\
& + {E_T} - E_T^* - E_T^*\ln \frac{{{E_T}}}{{E_T^*}} + B({I_T} - I_T^* - I_T^*\ln \frac{{{I_T}}}{{I_T^*}})
\end{split}
\end{equation}

with Lyaponuv derivative given as
\[\begin{array}{l}
\dot L = ({{\dot S}_B} - \frac{{S_B^*}}{{{S_B}}}{{\dot S}_B}) + ({{\dot E}_B} - \frac{{{E_B}^*}}{{{E_B}}}{{\dot E}_B}) + A({\mathop I\limits^. _B} - \frac{{I_B^*}}{{{I_B}}}{\mathop I\limits^. _B}) + ({{\dot S}_B} - \frac{{S_B^*}}{{{S_B}}}{{\dot S}_B}) + ({{\dot E}_T} - \frac{{{E_T}^*}}{{{E_T}}}{{\dot E}_T}) + B({\mathop I\limits^. _T} - \frac{{I_T^*}}{{{I_T}}}{\mathop I\limits^. _T})\\
\ \text{where} \ A = \,\left( {\frac{{{\beta _2}S_B^* + {\beta _3}S_T^*}}{{\sigma  + d}}} \right)\,\,\,and\,\,\,B = \left( {\frac{{{\beta _1}S_B^* + (\theta  + \lambda )S_T^*}}{\delta }} \right)
\end{array}\]

\begin{equation}
\begin{split}
\dot L = [({\tau _B} & - {\beta _1}{I_T}{S_B}  - {\beta _2}{I_B}{S_B} - d{S_B}) - \frac{{S_B^*}}{{{S_B}}}({\tau _B} - {\beta _1}{I_T}{S_B} - {\beta _2}{I_B}{S_B} - d{S_B})]\\ 
& + [({\beta _1}{I_T}{S_B} + {\beta _2}{I_B}{S_B} - ({\alpha _B} + d){E_B}) - \frac{{{E_B}^*}}{{{E_B}}}({\beta _1}{I_T}{S_B} + {\beta _2}{I_B}{S_B} - ({\alpha _B} + d){E_B})) \\
& + A({\alpha _B}{E_B} - (\sigma  + d + \mu ){I_B}) - \frac{{I_B^*}}{{{I_B}}}({\alpha _B}{E_B} - \sigma {I_B} - (d + \mu {I_B})) \\
& + [({\tau _T} - {\beta _3}{I_B}{S_T} - (\theta  + \lambda ){S_T}{I_T} - \delta {S_T}) - \frac{{S_B^*}}{{{S_B}}}({\tau _T} - {\beta _3}{I_B}{S_T} - (\theta  + \lambda ){S_T}{I_T} - \delta {S_T})]\\
& + [({\beta _3}{I_B}{S_T} + (\theta  + \lambda ){S_T}{I_T} - (\delta  + {\alpha _T}){E_T} - \frac{{{E_T}^*}}{{{E_T}}}({\beta _3}{I_B}{S_T} + (\theta  + \lambda ){S_T}{I_T} - (\delta  + {\alpha _T}){E_T})] \\
& + B[({\alpha _T}{E_T} - \delta {I_T}) - \frac{{I_T^*}}{{{I_T}}}({\alpha _T}{E_T} - \delta {I_T})]
\end{split}
\end{equation}

At steady states 
\begin{equation*}
\begin{split}
{\tau _B} &= {\beta _1}I_T^*S_B^* + {\beta _2}I_B^*S_B^* + dS_B^*\\
{\tau _T} &= {\beta _3}I_B^*S_T^* + (\theta  + \lambda )S_T^*I_T^* - \delta S_T^*\\
(\sigma+d)&=\frac{\alpha_B E_B^*}{I_B^*}\\
\delta &=  \frac{\alpha_T E_T^*}{I_T^*}
\end{split}
\end{equation*}
Substituting the values of $\tau_B$ and $\tau _T$ at steady states gives 

\begin{equation}
\begin{split}
\dot L = [({\beta _1}I_T^*S_B^* & + {\beta _2}I_B^*S_B^* + dS_B^* - {\beta _1}{I_T}{S_B} - {\beta _2}{I_B}{S_B} - d{S_B}) - \frac{{S_B^*}}{{{S_B}}}({\beta _1}I_T^*S_B^* + {\beta _2}I_B^*S_B^* \\
& + dS_B^* - {\beta _1}{I_T}{S_B} - {\beta _2}{I_B}{S_B} - d{S_B})] + [({\beta _1}{I_T}{S_B} + {\beta _2}{I_B}{S_B} - ({\alpha _B} + d){E_B})\\
 & - \frac{{{E_B}^*}}{{{E_B}}}({\beta _1}{I_T}{S_B} + {\beta _2}{I_B}{S_B} - ({\alpha _B} + d){E_B})] + A({\alpha _B}{E_B} - (\sigma  + d + \mu ){I_B}) - \frac{{I_B^*}}{{{I_B}}}({\alpha _B}{E_B} - \sigma {I_B} - (d + \mu {I_B})) \\
 & + [({\beta _3}I_B^*S_T^* + (\theta  + \lambda )S_T^*I_T^* - \delta S_T^* - {\beta _3}{I_B}{S_T} - (\theta  + \lambda ){S_T}{I_T} - \delta {S_T})\\
& - \frac{{S_B^*}}{{{S_B}}}({\beta _3}I_B^*S_T^* + (\theta  + \lambda )S_T^*I_T^* - \delta S_T^* - {\beta _3}{I_B}{S_T} - (\theta  + \lambda ){S_T}{I_T} - \delta {S_T})] \\
& + [({\beta _3}{I_B}{S_T} + (\theta  + \lambda ){S_T}{I_T} - (\delta  + {\alpha _T}){E_T} - \frac{{{E_T}^*}}{{{E_T}}}({\beta _3}{I_B}{S_T} + (\theta  + \lambda ){S_T}{I_T} \\
& - (\delta  + {\alpha _T}){E_T})] + B[({\alpha _T}{E_T} - \delta {I_T}) - \frac{{I_T^*}}{{{I_T}}}({\alpha _T}{E_T} - \delta {I_T})]
\end{split}
\end{equation}
Simplifying , we have

\begin{equation}
\begin{split}
\dot L = {\beta _1}I_T^*S_B^* & + {\beta _2}I_B^*S_B^* + dS_B^* - {\beta _1}{I_T}{S_B} - {\beta _2}{I_B}{S_B} - d{S_B} - \frac{{S_B^*}}{{{S_B}}}{\beta _1}I_T^*S_B^* - \frac{{S_B^*}}{{{S_B}}}{\beta _2}I_B^*S_B^* \\
& - \frac{{S_B^*}}{{{S_B}}}dS_B^* + S_B^*{\beta _1}{I_T}{S_B} + S_B^*{\beta _2}{I_B} + dS_B^* + {\beta _1}{I_T}{S_B} + {\beta _2}{I_B}{S_B} - ({\alpha _B} + d){E_B} - \frac{{{E_B}^*}}{{{E_B}}}{\beta _1}{I_T}{S_B} \\
& - \frac{{{E_B}^*}}{{{E_B}}}{\beta _2}{I_B}{S_B} + ({\alpha _B} + d){E_B}^*
 + [\,\left( {\frac{{{\beta _2}S_B^* + {\beta _3}S_T^*}}{{\sigma  + d}}} \right)\,({\alpha _B}{E_B} - (\sigma  + d + \mu ){I_B}] - \frac{{I_B^*}}{{{I_B}}}{\alpha _B}{E_B} \\
 & + \sigma I_B^* + (d + \mu )I_B^* + {\beta _3}I_B^*S_T^* + (\theta  + \lambda )S_T^*I_T^* - \delta S_T^* - {\beta _3}{I_B}{S_T} - (\theta  + \lambda ){S_T}{I_T} - \delta {S_T} \\
 & - \frac{{S_B^*}}{{{S_B}}}{\beta _3}I_B^*S_T^* - \frac{{S_B^*}}{{{S_B}}}(\theta  + \lambda )S_T^*I_T^* + \frac{{S_B^*}}{{{S_B}}}\delta S_T^* + \frac{{S_B^*}}{{{S_B}}}{\beta _3}{I_B}{S_T} + \frac{{S_B^*}}{{{S_B}}}(\theta  + \lambda ){S_T}{I_T} + \frac{{S_B^*}}{{{S_B}}}\delta {S_T} \\
 & + {\beta _3}{I_B}{S_T} + (\theta  + \lambda ){S_T}{I_T} - (\delta  + {\alpha _T}){E_T} - \frac{{{E_T}^*}}{{{E_T}}}{\beta _3}{I_B}{S_T} - \frac{{{E_T}^*}}{{{E_T}}}(\theta  + \lambda ){S_T}{I_T} + (\delta  + {\alpha _T}){E_T}^* \\
 & + \left( {\frac{{{\beta _1}S_B^* + (\theta  + \lambda )S_T^*}}{\delta }} \right)[{\alpha _T}{E_T} - \delta {I_T} - \frac{{I_T^*}}{{{I_T}}}{\alpha _T}{E_T} + \delta I_T^*]
\end{split}
\end{equation}

Collecting terms with

$dS_B^*,\delta S_T^* ,{\beta _1}S_B^*I_T^*,{\beta _2}S_B^*I_B^*,{\beta _3}S_T^*I_B^*$ and $(\theta  + \lambda )S_T^*I_T^*$ gives,

\begin{equation}
\begin{split}
\mathop L\limits^.  = dS_B^*\left[ {2 - \frac{{{S_B}}}{{S_B^*}} - \frac{{S_B^*}}{{{S_B}}}} \right] + \delta S_T^*\left[ {2 - \frac{{{S_T}}}{{S_T^*}} - \frac{{S_T^*}}{{{S_T}}}} \right] + {\beta _1}S_B^*I_T^*\left[ {3 - \frac{{S_B^*}}{{{S_B}}} - \frac{{{S_B}E_B^*{I_T}}}{{S_B^*{E_B}I_T^*}} - \frac{{I_T^*{E_B}}}{{{I_T}E_B^*}}} \right]\\
 + {\beta _2}S_B^*I_B^*\left[ {3 - \frac{{S_B^*}}{{{S_B}}} - \frac{{{S_B}{I_B}{E_B}^*}}{{S_B^*I_B^*{E_B}}} - \frac{{I_B^*{E_B}}}{{{I_B}E_B^*}}} \right] + {\beta _3}S_T^*I_B^*\left[ {3 - \frac{{S_T^*}}{{{S_T}}} - \frac{{{I_B}{S_T}{E_T}^*}}{{I_B^*S_T^*{E_T}}} - \frac{{I_B^*{E_T}}}{{{I_B}E_T^*}}} \right]\\
 + (\theta  + \lambda )S_T^*I_T^*\left[ {3 - \frac{{S_T^*}}{{{S_T}}} - \frac{{{S_T}{I_T}{E_T}^*}}{{S_T^*I_T^*{E_T}}} - \frac{{I_T^*{E_T}}}{{{I_T}E_T^*}}} \right]
\end{split}
\end{equation}

Finally since the arithmetic mean exceeds the geometric mean, it follows that,
\[dS_B^*\left[ {2 - \frac{{{S_B}}}{{S_B^*}} - \frac{{S_B^*}}{{{S_B}}}} \right] \le 0\] ,  \[\delta S_T^*\left[ {2 - \frac{{{S_T}}}{{S_T^*}} - \frac{{S_T^*}}{{{S_T}}}} \right] \le 0\],  \[{\beta _1}S_B^*I_T^*\left[ {3 - \frac{{S_B^*}}{{{S_B}}} - \frac{{{S_B}E_B^*{I_T}}}{{S_B^*{E_B}I_T^*}} - \frac{{I_T^*{E_B}}}{{{I_T}E_B^*}}} \right] \le 0\]
\[{\beta _2}S_B^*I_B^*\left[ {3 - \frac{{S_B^*}}{{{S_B}}} - \frac{{{S_B}{I_B}{E_B}^*}}{{S_B^*I_B^*{E_B}}} - \frac{{I_B^*{E_B}}}{{{I_B}E_B^*}}} \right] \le 0\],  \[{\beta _3}S_T^*I_B^*\left[ {3 - \frac{{S_T^*}}{{{S_T}}} - \frac{{{I_B}{S_T}{E_T}^*}}{{I_B^*S_T^*{E_T}}} - \frac{{I_B^*{E_T}}}{{{I_B}E_T^*}}} \right] \le 0\]
Since all the model parameters are non-negative, it follows that $L\leq 0$ for $R_0>1$. Thus, $L$ is a Lyaponuv function for the system of model (\ref{avian1}). Furthermore, we note that $L ̇= 0$ holds only at $E_E^*$. By Lasalle’s invariant principle, every solution to the system (\ref{avian1}), with the initial conditions in $\Omega$, approaches $E_E^*$ as $t \rightarrow \infty$ if $R_0>$1. Hence, the endemic equilibrium $E_E^*$ is globally asymptotically stable in $\Omega$ if $R_0>$1. 

\end{proof}

The epidemiological implication of the result is that Avian Spirochaetosis will establish itself (be endemic) in the poultry whenever $R_0> 1$.

%%%%%%%%%%%%%%%%%%%%%%%%%%%%%%%%%%%%%%%%%%%%%%%%%%%%%%%%%%%%
%%%%%%%%%%%%%%%%%%%%%%%%%%%%%%%%%%%%%%%%%%%%%%%%%%%%%%%%%%%%

\section{Analysis of the Optimal Control Model for Avian Spirochaetosis Model}
The associated forces of infections are reduced by the controls ${u_1}(t)\,,\,{u_2}(t)\,\,and\,\,{u_3}(t)$. The control $u_1(t)$ represents the effort to reduce the number of latently infected birds such as through bio-security measures that involves regular dispose of bird's faeces and general poultry sanitation. The control variable ${u_2}(t)$represents the use of antibiotics such as penicillin to minimize the number of infectious birds. Finally, we describe the role of the third control variable ${u_3}(t)$. The control variable ${u_3}(t)$ represents the level of insecticide such as Permerthrin used for tick control, administered at tick breeding sites to eliminate specific breeding areas. 
Our aim is to minimize the exposed and the infectious bird population, the total number of tick population and the cost of implementing the control by using possible minimal control variables ${u_i}(t)$ for $i = 1,\,2,\,3.$

The objective function is defined as 
\begin{equation}
J({u_1},{u_2},{u_3}) = \int_0^T {({C_1}} {E_B} + {C_2}{I_B} + {C_3}{N_T} + \frac{1}{2}({D_1}u_1^2 + {D_2}u_2^2 + {D_3}u_3^2)dt 
\end{equation}

Subject to the state system (\ref{control1}) and initial conditions (\ref{initialcond}). The quantities ${C_1},{C_2},{C_3},{D_1},{D_2},{D_3}$ are positive weight constants. The terms ${C_1}{E_B},\,{C_2}{I_B}$ and ${C_3}{N_T}$ denote the cost associated in reducing the exposed, infectious and the total tick population respectively. Also ${D_1}u_1^2,\,{D_2}u_2^2$ and ${D_3}u_3^2$ represent the cost associated with the control measures. The purpose is then to find an optimal control triplet $u_1^*\,,\,u_2^*\,,\,u_3^*$ which satisfy
\begin{equation}
J({u_1},\,{u_2},\,{u_3}) = {\min _{({u_1},\,{u_2},{u_3}) \in U}}J({u_1},\,{u_{2\,}},{u_3})\label{objfuncn}
\end{equation}
where,
$U = \left\{ \left( {u_1},\,{u_2},{u_3} \right)| {{u_i}(t):0 \le {u_i}(t) \le {m_i},\,0 \le t \le T,\,\,i = 1,2,3;\,\,u(t)} \ \text{is measurable.}\ \right\}$ 
	
\subsection{Existence of Optimal Control}
\begin{theorem}
Consider the objective function of (\ref{objfuncn}) with $({u_1},\,{u_2},{u_3}) \in U$ Subject to the control system of (\ref{control1}), there exist ${U^*} = ({u_1},\,{u_2},{u_3}) \in U$ such that \[{\min _{({u_1},\,{u_{2,\,}}{u_3}) \in U}}J({u_1},\,{u_{2\,}},{u_3}) = J(u_1^*\,,\,u_2^*\,,\,u_3^*)\]
 \end{theorem}
 	
\begin{proof}
The existence of the optimal control can be obtained using a result by Fleming and Rishal \cite{Fleming} and used in Nordin et al \cite{Nordin}. Checking the following conditions;
\begin{itemize}
\item[(i)]	 From (\ref{control1}), it follows that the set of controls and corresponding state variables are non-empty.
\item[(ii)]	The control set $U = \left\{ {u:\,u\,\,\text{are} \ \,\, \text{measurable}\,,\,\,0 \le u(t) \le {m_i},\,t \in [0,\,T]\,} \right\}$ is convex and closed by definition
\item[(iii)]	The right hand side of the state system (\ref{control1}) is bounded above by a sum of bounded control and state, and can be written as a linear function of u with coefficients depending on time and state.
\item[(iv)] The integrand of the objective functional ${C_1}{E_B} + \,{C_2}{I_B} + \,{C_3}{N_T} + {\rm{\raise.5ex\hbox{$\scriptstyle 1$}\kern-.1em/
 \kern-.15em\lower.25ex\hbox{$\scriptstyle 2$} }}({D_1}u_1^2 + {D_2}u_2^2 + {D_3}u_3^2)$ is convex on u. There exist
	\tcr{ ${r_1},\,{r_2} > 0$ and $\pi  > 1$  satisfying }
	$${C_1}{E_B} + \,{C_2}{I_B} + \,{C_3}{N_T} + {D_1}u_1^2 + {D_2}u_2^2 + {D_3}u_3^2 \ge {r_1}{({\left| {{u_1}} \right|^2} + {\left| {{u_2}} \right|^2} + {\left| {{u_3}} \right|^2})^{\pi/2}} - {r_2}$$ 
\end{itemize}
	
since the state variables are bounded. Hence we can conclude that there exists an optimal control, which completes the existence of an optimal control.
\end{proof}

\subsection{Characterization of Optimal Control}
Pontryagin’s maximum principle is used to derive the necessary conditions for the optimal control triplet. We shall now characterize the optimal control triplet $u_1^*,u_2^*,u_3^*$, which accomplish the set objectives and the corresponding states $(S_B^*,E_B^*, I_B^*, R^*, S_T^*, E_T^*, I_T^*)$ using the pontryagin’s maximum principle. The Hamiltonian is defined as follows;

\begin{equation}
\begin{split}
H = {C_1}{E_B} &+ \,{C_2}{I_B} + \,{C_3}{N_T} + \frac{1}{2}({D_1}u_1^2 + {D_2}u_2^2 + {D_3}u_3^2) + {\lambda _1}[{\tau _B}{N_B} - (1 - {u_1}){B_1}{I_T}{S_B} - (1 - {u_1}){B_2}{I_B}{S_B} - d{S_B}]\\ 
& + {\lambda _2}[(1 - {u_1}){B_1}{I_T}{S_B} + (1 - {u_2}){B_2}{I_B}{S_B} - {\alpha _B}{E_B} - d{E_B}] + {\lambda _3}[{\alpha _B}{E_B} - {\mu _B}{I_B} - d{I_B}] \\
& + {\lambda _6}[{B_3}{I_B}{S_T} + \theta {I_T}{S_T} + \lambda {S_T}{I_T} + \delta {E_T} - {\alpha _T}{E_T}] + {\lambda _7}[{\alpha _T}{E_T} - \delta {I_T} - {u_1}{I_T}] + {\lambda _4}[{u_2}{I_B} - dR]\\
& + {\lambda _5}[{\tau _T}{N_T}(1 - {u_3}) - {B_3}{I_B}{S_T} - \theta {I_T}{S_T} - \lambda {S_T}{I_T} - \delta {S_T}]
\end{split}\label{eq3.4}
\end{equation}

\begin{theorem}
There exist an optimal control $u_1^*,u_2^*, u_3^*$ and the corresponding state solutions 
$(S_B^*,E_B^*, I_B^*, R_B^*, S_T^*, E_T^*, I_T^*)$  of the system (\ref{avian1}), that minimizes $J(u_1,u_2,u_3)$ over $U$. Furthermore, there exist adjoint functions ${\lambda _i}$ ,for $i = \,1,\,2,\,3,\,4,\,5,\,6,\,7$ such that;

\begin{equation}
\begin{split}
{\lambda '_1} &= {\lambda _1}[{\tau _B} + (1 - {u_1}){B_1}{I_T} + (1 - {u_1}){B_2}{I_B} + d{S_B}] - {\lambda _2}[(1 - {u_1}){B_1}{I_T} + (1 - {u_1}){B_2}{I_B}\\
{\lambda '_2} &=  - {C_1} + {\lambda _1}[{\alpha _B}{d}] - {\lambda _3}[{\alpha _B}]\\
{\lambda '_3} & =  - {C_2} + {\lambda _1}[(1 - {u_1}){B_2}{S_B}] - {\lambda _2}[(1 - {u_1}){B_2}{S_B}] + {\lambda _3}[d + \mu  + {u_3}] - {\lambda _4}[{u_2}] + {\lambda _5}[{B_3}{S_T}] - {\lambda _6}[{B_3}{S_T}]\\ 
{\lambda '_4} &= d{\lambda _4}\\
{\lambda '_5} & =  - {C_3} + {\lambda _5}[{B_3}{I_B} + \theta {I_T} + \lambda {I_T} + \delta  - {\tau _T}(1 - {u_3})] - {\lambda _6}[{B_3}{I_B} - \theta {I_T} - \lambda {I_T}]\\
{\lambda '_6} &=  - {C_3} - {\lambda _5}[{\tau _T}(1 - {u_3})] + {\lambda _6}[{\alpha _T} + \delta ] - {\lambda _7}[{\alpha _T}]\\
{\lambda '_7} & =  - {C_3} - {\lambda _1}[(1 - {u_1}){B_1}{S_B}] - {\lambda _2}[(1 - {u_1}){B_1}{S_B}] - {\lambda _5}[{\tau _T}(1 - {u_3}) - \theta {S_T} - \lambda {S_T}] - {\lambda _6}[\theta {S_T} + \lambda {S_T} + \delta ] + \delta {\lambda _7}
\end{split}
\end{equation}    

with the transversality condition of ${\lambda _i}(T) = 0\,\,,\,i\, = \,1,\,2,\,3,\,4,\,5,\,6,\,7$. The optimality controls are given by

\begin{equation}
\begin{split}
u_1^* & = \max\left\{0 ,\min\left(m_1,\frac{(B_1 I_T S_B+ B_2 I_B S_B)(\lambda_2- \lambda_1)}{D_1} \right) \right\}\\
u_2^* & = \max \left\{0 ,\min \left(m_2,  \frac{(\lambda_3- \lambda_4)I_B)}{D_2} \right) \right\}\\                                                 
u_3^* & = \max \left\{0 ,\min \left(m_3,\frac{\lambda_5 \tau_T N_T)}{D_3} \right) \right\}
\end{split}
\end{equation}    
\end{theorem}
	
\begin{proof}
The form of the adjoint functions and tranversality condition are standard results from pontryagin’s maximum principle. The Hamiltonian is differentiated with respect to the states $S_B, E_B, I_B, R_B, S_T, E_T, I_T$ 
respectively, which results in the following adjoint functions.
\[{\lambda '_1}(t) =  - \frac{{\partial H}}{{\partial {S_B}}},\,\,{\lambda '_2}(t) =  - \frac{{\partial H}}{{\partial {E_B}}},\,\,{\lambda '_3}(t) =  - \frac{{\partial H}}{{\partial {I_B}}},\,\,{\lambda '_4}(t) =  - \frac{{\partial H}}{{\partial {R_B}}},\,\,{\lambda '_5}(t) =  - \frac{{\partial H}}{{\partial {S_T}}},\,\,{\lambda '_6}(t) =  - \frac{{\partial H}}{{\partial {E_T}}},\,\,{\lambda '_7}(t) =  - \frac{{\partial H}}{{\partial {I_T}}}\]
With ${\lambda _i}(T) = 0\,\,,\,i\, = \,1,\,2,\,3,\,4,\,5,\,6,\,7$
The characterization of the optimal control is obtained by solving the equations
	\[\frac{{\partial H}}{{\partial {u_1}}} = {D_1}{u_1}(t) + {\lambda _1}[{B_1}{I_T}{S_B} + {B_2}{I_B}{S_B}] - {\lambda _2}[{B_1}{I_T}{S_B} + (1 - {u_1}){B_2}{I_B}{S_B}] = 0\]
	\[\frac{{\partial H}}{{\partial {u_2}}} = {D_2}{u_2}(t) + ({\lambda _4} - {\lambda _3}){I_B} = 0\]
         \[\frac{{\partial H}}{{\partial {u_3}}} = {D_3}{u_3}(t) - {\lambda _5}{\tau _T}{N_T} = 0\]
Solving for each of the optimal control we have,
\[u_1^* = \frac{{[{B_1}{I_T}{S_B} + {B_2}{I_B}{S_B}]({\lambda _2} - {\lambda _1})}}{{{D_1}}}\]
\[u_2^* = \frac{{({\lambda _3} - {\lambda _4}){I_B}}}{{{D_2}}}\]
\[u_3^* = \frac{{{\lambda _5}{\tau _T}{N_T}}}{{{D_3}}}\]	
Therefore, the optimal control $u_1^*,\,u_2^*,u_3^*$ exists and is characterized by the following: \\
\begin{equation}
\begin{split}
u_1^* & = \max \left\{ {0,\,\,\min \left( {{m_1},\,\frac{{({B_1}{I_T}{S_B} + {B_2}{I_B}{S_B})({\lambda _2} - {\lambda _1})}}{{{D_1}}}} \right)} \right\}\\
u_2^* & = \max \left\{ {0,\,\,\min \left( {{m_2},\,\frac{{({\lambda _3} - {\lambda _4}){I_B}}}{{{D_2}}}} \right)} \right\}\\
u_3^* & = \max \left\{ {0,\,\,\min \left( {{m_3},\,\frac{{({\lambda _5}{\tau _T}{N_T})}}{{{D_3}}}} \right)} \right\}
\end{split}
\end{equation} 
	
This implies that the optimal effort necessary to reduce avian Spirochaetosis disease is   
\begin{equation}
\begin{split}
u_1^* &= \frac{{[{B_1}{I_T}{S_B} + {B_2}{I_B}{S_B}]({\lambda _2} - {\lambda _1})}}{{{D_1}}}\\
u_2^* &= \frac{{({\lambda _3} - {\lambda _4}){I_B}}}{{{D_2}}}\\
u_3^* &= \frac{{{\lambda _5}{\tau _T}{N_T}}}{{{D_3}}}
\end{split}
\end{equation}
\end{proof}

\section{Results}
The objective of our numerical simulation experiments will be to better understand the dynamics involved in the Avian Spirochaetosis infection and the effect of control measures inclusion. Considering the estimated value of parameters in Table (\ref{tablea}),computation is done with MATLAB and the results are presented as follows:

\subsection*{Temporal dynamics of the Spirochaetosis model:}
This first simulation experiments (see Figure 2) reveals the temporal dynamics of the state variables using the base parameter values. Here we observe that on the long run; (i) the population of infected tick dies out completely while between the periods of $2$ to $4$ years the populations of the infected birds become evenly distributed, (ii)	the rate at which the infected birds become recovered is on a high increase, (iii) the exposed tick dies out while that of the bird population increases between the periods of $0$ to $16$ years and then attains a plateau.

%%%%%%%%%%%%%%%%%%%
\begin{figure}[h!]
\centering
\begin{tabular}{c}
\includegraphics[width=0.66\textwidth]{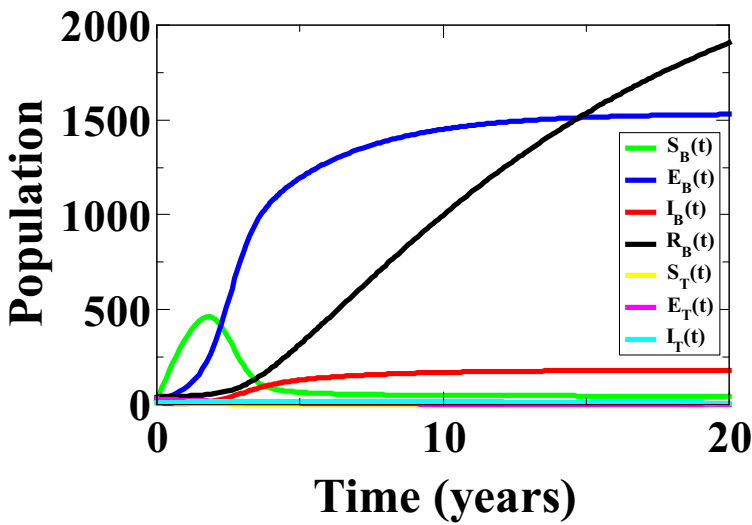}
\end{tabular}
\caption[Temporal dynamics of the Spirochaetosis model:]{\textbf{Temporal dynamics of the Spirochaetosis model:} Presented in this figure is the population dynamics of the state variables ($S_B, E_B, I_B, R, S_T, E_T, I_T$) using the base parameter values.}\label{evol1}
\end{figure}

\subsection*{Influence of parameter values on temporal dynamics of the exposed and infected birds and ticks population:}
In Figure 3, we studied the effect of varying the rate of growth and decay of the state variables.  We started by varying $\alpha_B$ and $d$ in order to investigate its impact on the population of infected birds in the model. The progression rate from exposed to infectious class among birds $\alpha_B$ is varied from $0.1$ to $0.5$ increasing by one order of magnitude. The results reveals that as the rate of progression $\alpha_B$ increases, the infectious bird population also increases as shown in the line plots in Figure \ref{evol1}a from black to red.  The second aspect of this simulation experiment is to vary the natural death rate $d$ from $0.5$ to $1.5$, increasing the values by $100\%$ as presented in Figure \ref{evol1}b. The results here show that increasing the death rate will significantly reduce the infectious class of birds. Following in Figure \ref{evol1}c, we study the population dynamics of infectious class of of ticks when the progression rate $\alpha_T$ from exposed to infectious class among the tick and the death rate of ticks $\delta$ are varied, this is presented in black and red lines respectively. The progression rate $\alpha_T$ is varied as $\{0.1, 0.2, 0.3 \}$ and the death rate $\delta$ is varied as $\{0.08, 0.16, 0.24, 0.32, 0.4\}$. There is a faster decline of the tick infectious class population when $\delta$ is decreased than when $\alpha_T$ is decreased. The final part in this set of simulation experiment investigates the effect of tick death rate $\delta$ on the population of the exposed class of ticks. The death rate $\delta$ is varied as $\{0.08, 0.16, 0.24, 0.32, 0.4\}$. Decreasing the death rate obviously decreases the exposure class of the tick population as shown in Figure \ref{evol1}d from black to red.

%%%%%%%%%%%%%%%%%%%
\begin{figure}[h!]
\centering
\begin{tabular}{cc}
\includegraphics[width=0.46\textwidth]{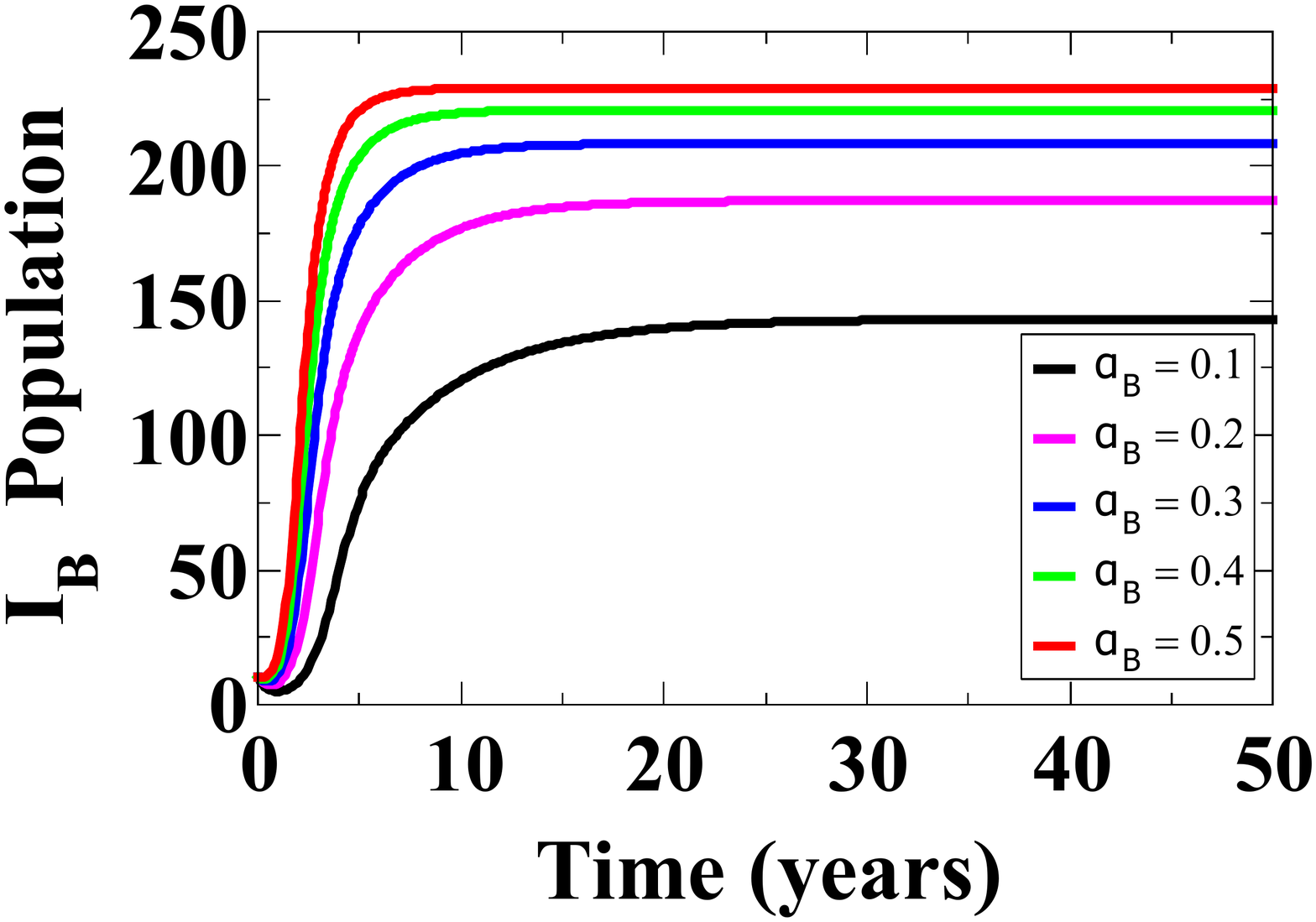}&
\includegraphics[width=0.46\textwidth]{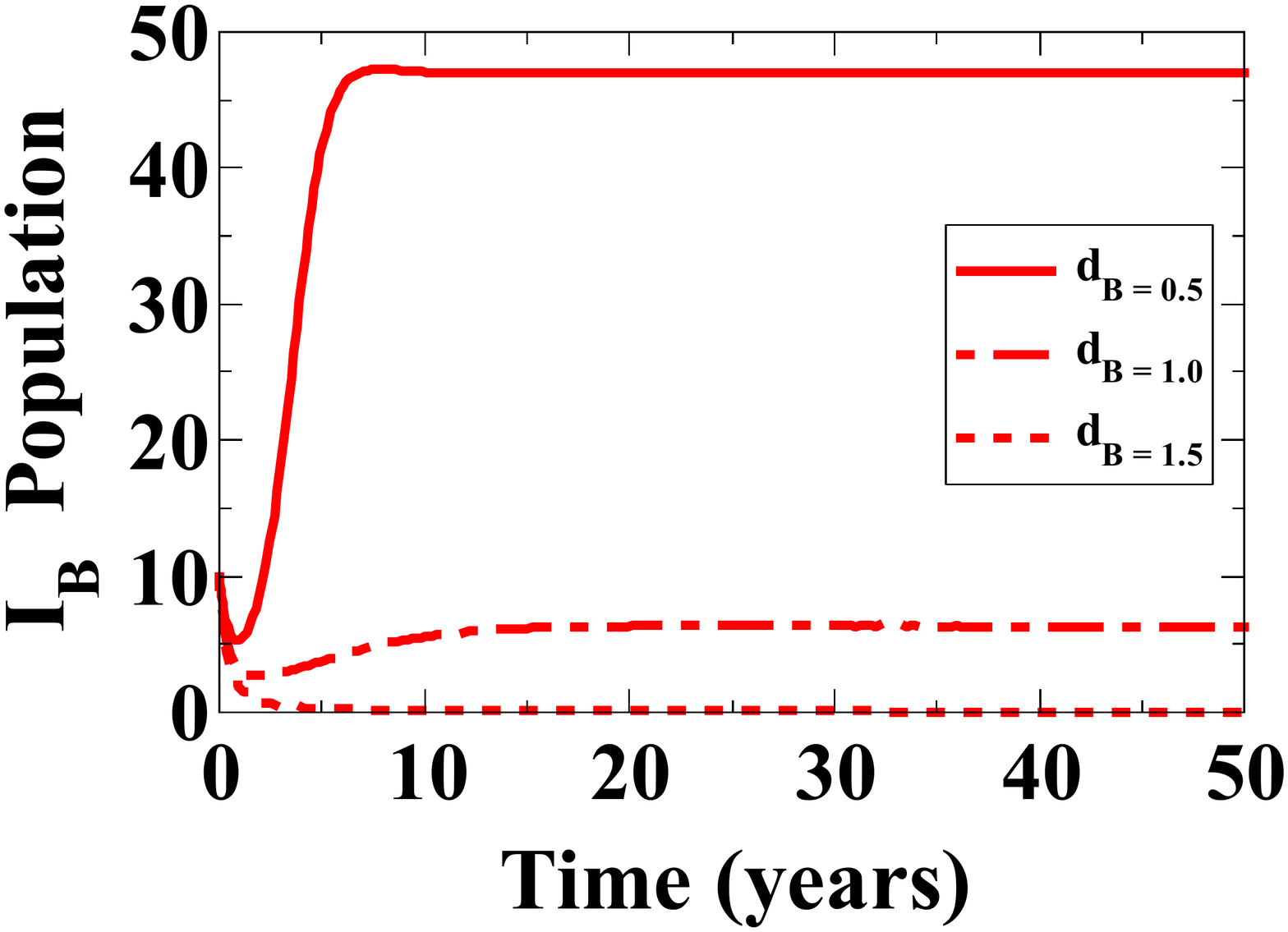}\\
(a) & (b)\\
\includegraphics[width=0.465\textwidth]{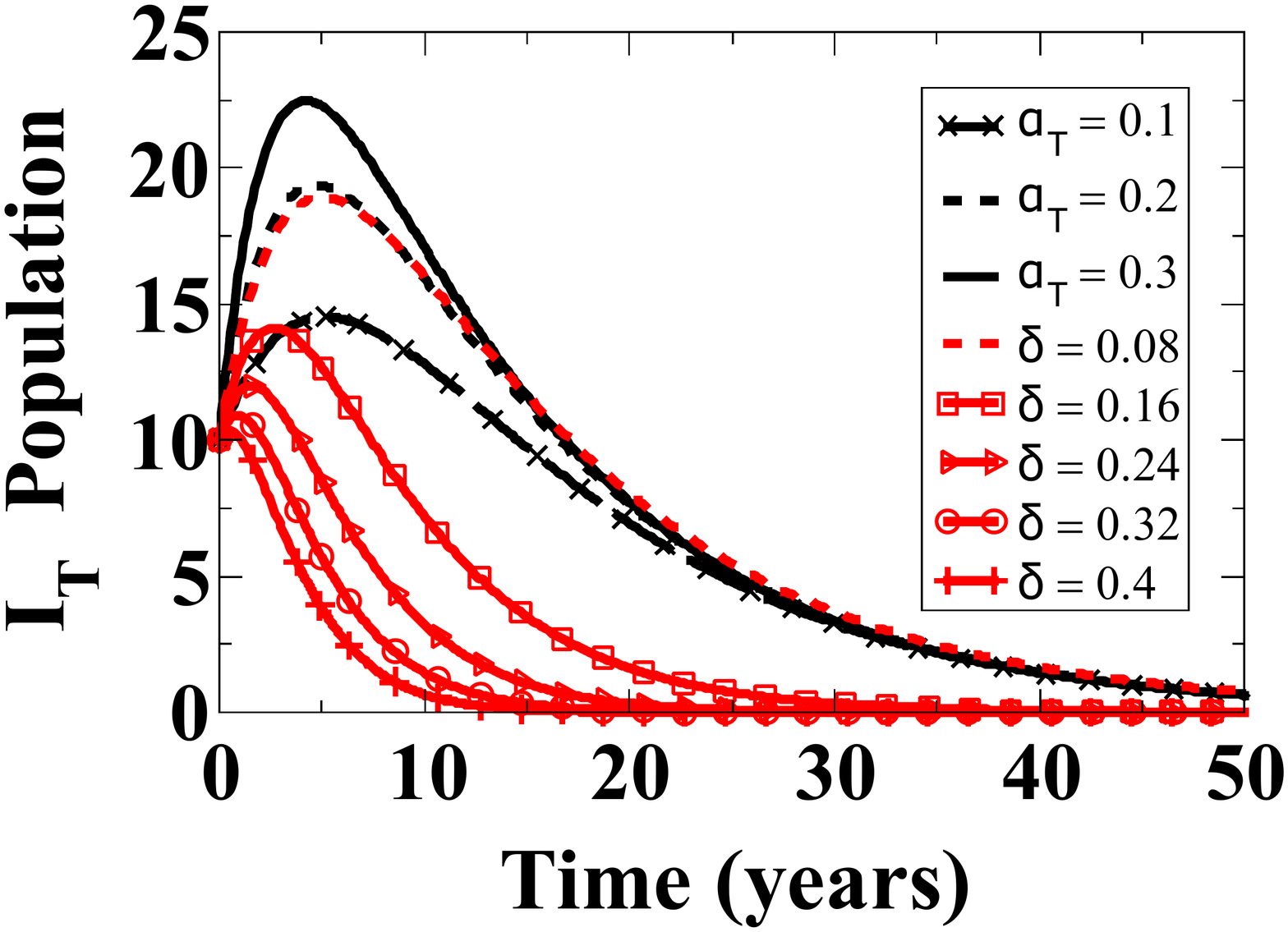}&
\includegraphics[width=0.46\textwidth]{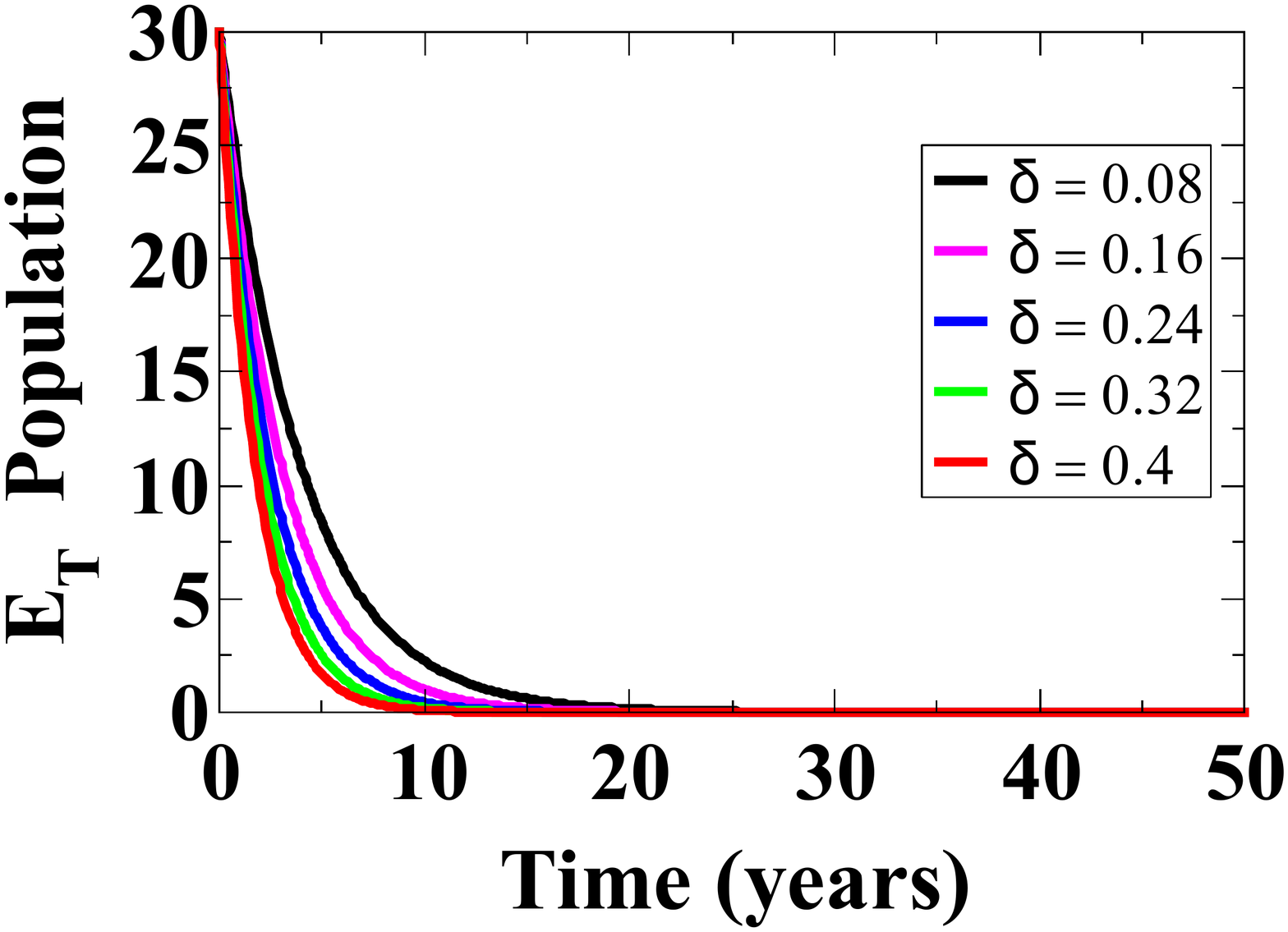}\\
(c) & (d)
\end{tabular}
\caption[Influence of parameter values on temporal dynamics of the exposed and infected birds and ticks population:]{\textbf{Influence of parameter values on temporal dynamics of the exposed and infected birds and ticks population:} Presented in this figure are (a)the population dynamics of infected birds when $\alpha_B$ is increased from $0.1$ to $0.5$ by one other of magnitude, (b)the population dynamics of infected birds when $\delta_B$ is increased from $0.5$ to $1.5$ by $100\%$, (c)the population dynamics of infected ticks when $\alpha_T$ and $\delta$ are varied as shown in black and red lines respectively, where $\alpha_T=\{0.1,0.2, 0.3\}$ and $\delta=\{0.08,0.16, 0.24,0.32,0.4\}$ (d)population dynamics of exposed ticks when $\delta$ is varied, where $\delta=\{0.08,0.16, 0.24,0.32,0.4\}$.}\label{evol1}
\end{figure}
%%%%%%%%%%%%%%%%%%%%%%%%%%%%%%%%%%%%%%%%%%%%%%%%%%%%%%%%%%%

\subsection*{Influence of parameter values on temporal dynamics of the Spirochaetosis control model:}
In Figure 4, we investigated the effect of varying the controls in the Spirochaetosis control model. The first results in this set of simulation experiment is the population dynamics of the state variables ($S_B, E_B, I_B, R, S_T, E_T, I_T$) using the base parameter values and the controls $u_1,u_2$ and $u_3$ (see Figure 4a). Here we observe that on the long run; (i) the population of infected tick dies out almost completely while between the periods of $2$ to $4$ years the populations of the infected birds become evenly distributed. These observations are very similar to the case without controls except for slight difference in order of magnitude.(ii)	the rate at which the infected birds become recovered goes to zero \textbf{( I am not sure if this makes sense)}, (iii) Similar to the case without controls, the exposed tick dies out while that of the bird population increases between the periods of $0$ to $16$ years and then attains a plateau.
The second results in this simulation experiment investigates the effect of varying the control $u_2$ on the population of the susceptible birds.
Here $u_2$ is varied as $\{0.2, 0.6, 1.0\}$ (see Figure 4b). The results show that increasing the value of $u_2$ gives the subsceptible birds population a chance of attaining a plateau after about 25-30 years. The next results, presented in Figure 4c reveals the population dynamics of the exposed class of birds when the control $u_1$ is varied as $\{0.02, 5, 20, 25\}$. Varying $u_1$ from $0.02$ to $4.8$ did not produce any significant difference, albeit the variations $\{5,20,25\}$ show some differences (see Figure 4c). Increasing the control $u_1$ value increases the rate at which $E_B$ decays and increases the rate at which it attains plateau at the long run. The final results in this set of simulation results is presented in Figure 4d, this is the population dynamics of the exposed class of birds when the control $u_2$ is varied as $\{0.3, 0.6, 0.9, 1.0, 1.2\}$ (\textbf{I don't really understand why this result (figure 4d) is like this, please check if this makes sense}).

\begin{figure}[h!]
\centering
\begin{tabular}{cc}
\includegraphics[width=0.465\textwidth]{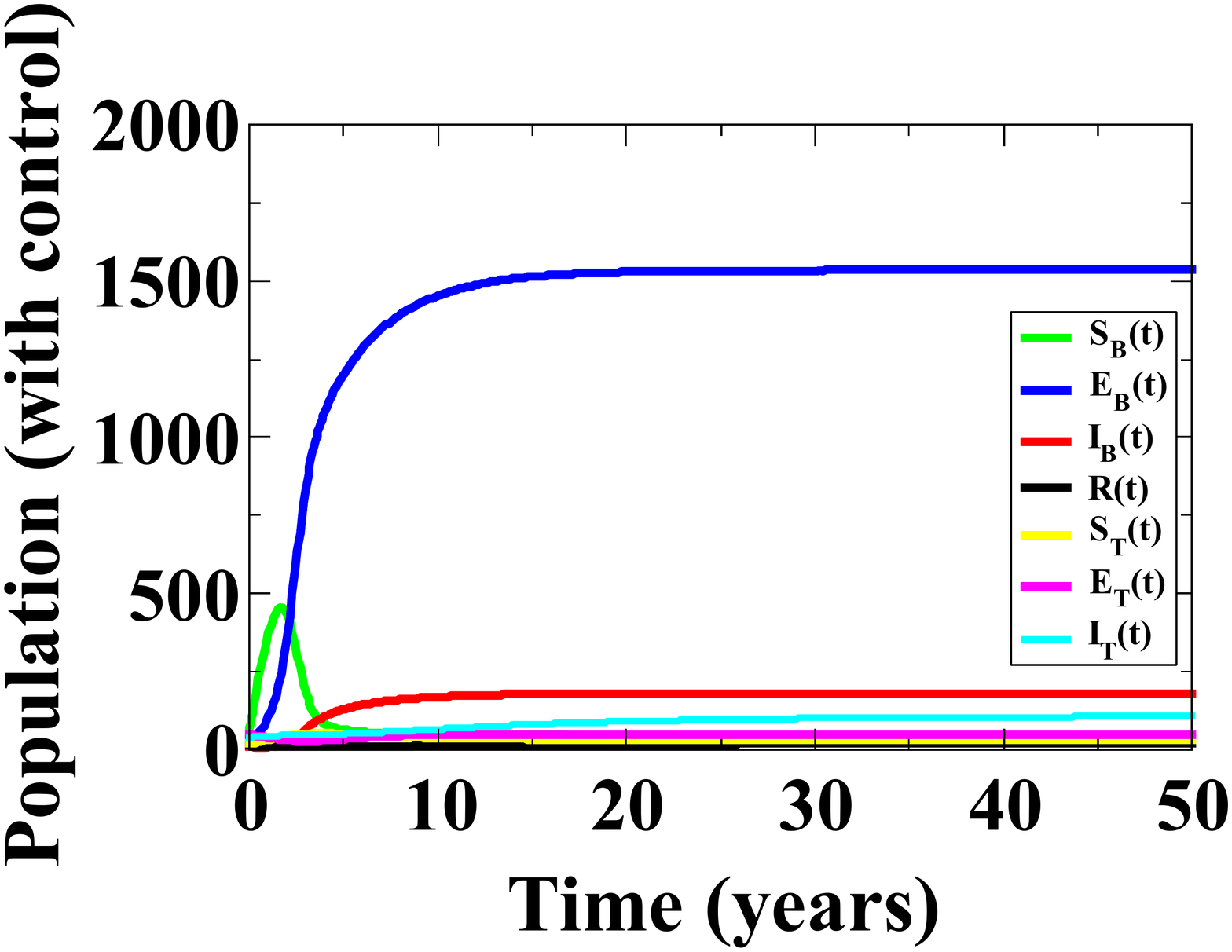}&
\includegraphics[width=0.46\textwidth]{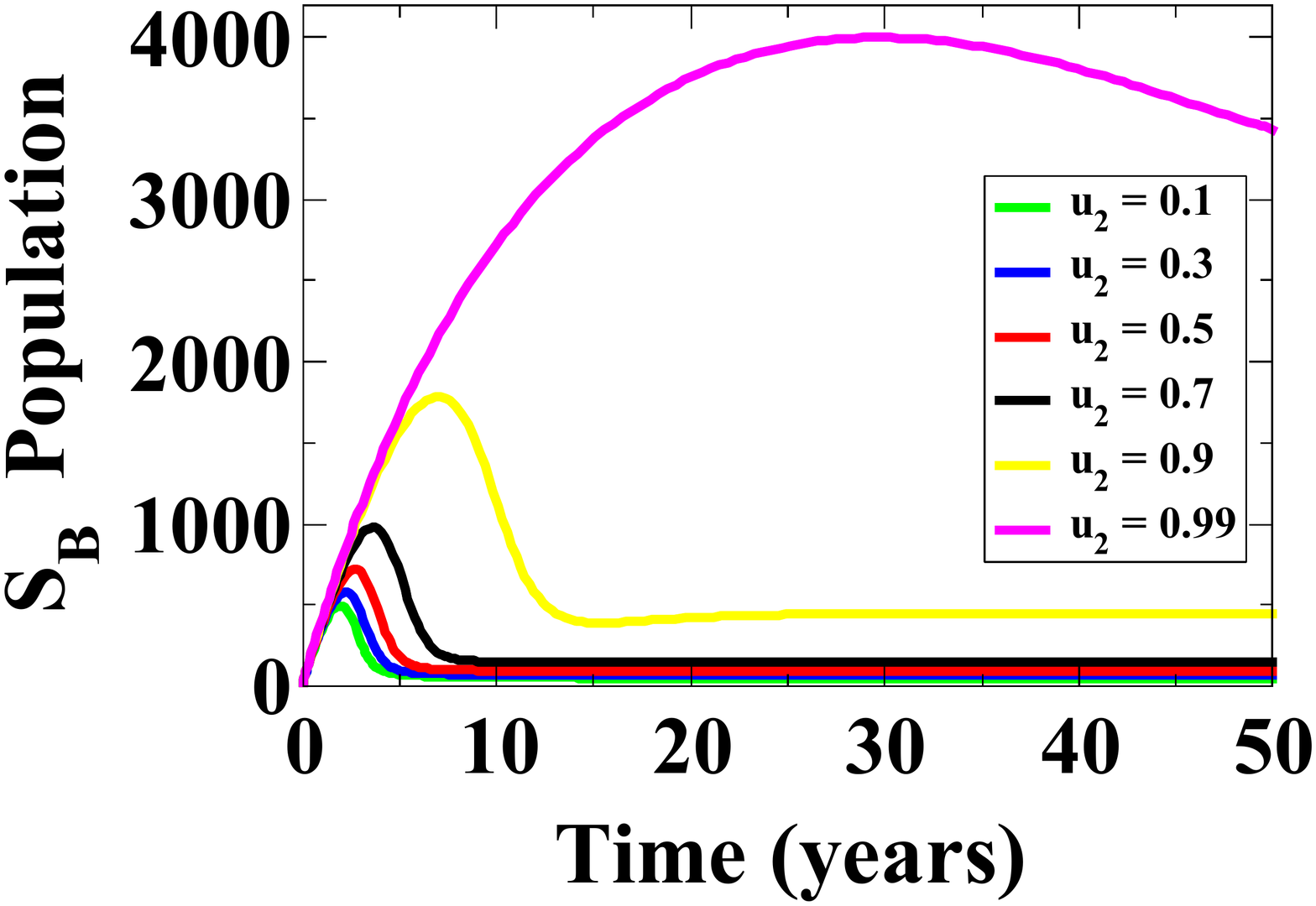}\\
(a) & (b)\\
\includegraphics[width=0.46\textwidth]{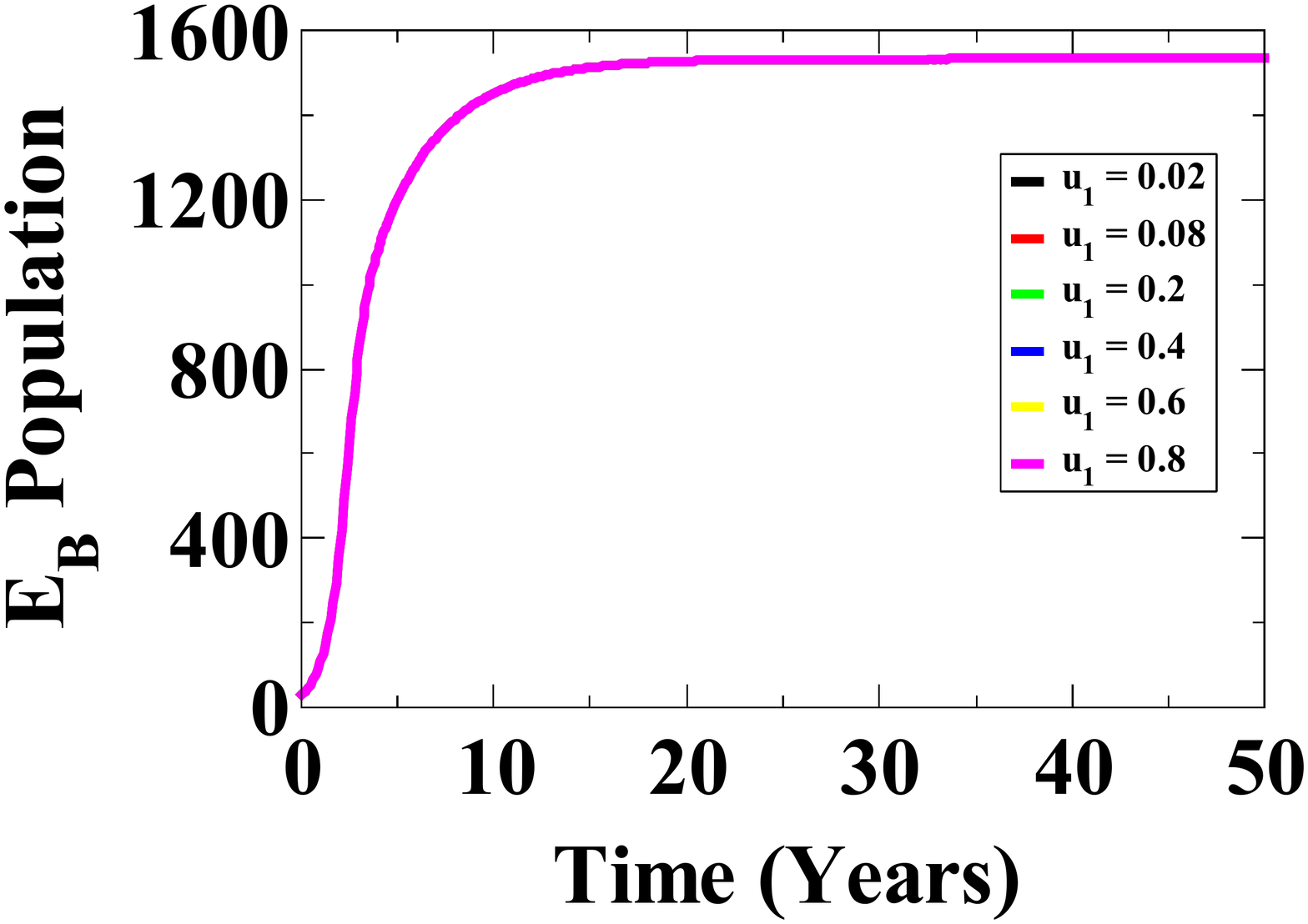}&
\includegraphics[width=0.46\textwidth]{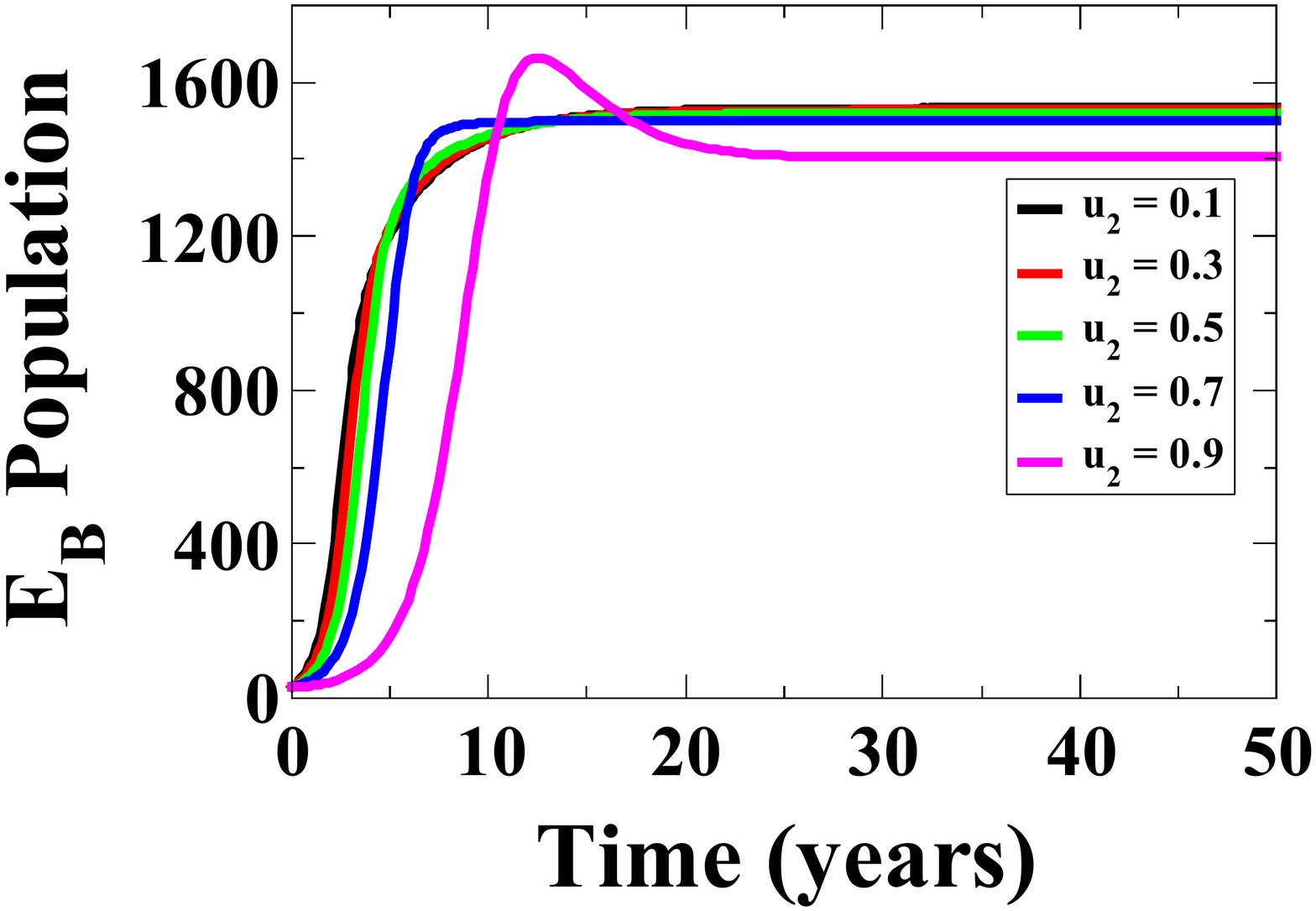}\\
 (c) & (d)
\end{tabular}
\caption[Influence of parameter values on temporal dynamics of the Spirochaetosis control model:]{\textbf{Influence of parameter values on temporal dynamics of the Spirochaetosis control model:} Presented in this figure are (a)population dynamics of the state variables ($S_B, E_B, I_B, R, S_T, E_T, I_T$) using the base parameter values and the controls $u_1,u_2$ and $u_3$. (b)the population dynamics of susceptible birds when $u_2$ is varied as $\{0.2, 0.6, 1.0\}$ which represented as black, red and blue lines respectively. (c)the population dynamics of the exposed class of birds $E_T$ when the control $u_1$ is varied as $\{0.02, 5, 20, 25\}$. (d)the population dynamics of the exposed class of birds when the control $u_2$ is varied as $\{0.3, 0.6, 0.9, 1.0, 1.2\}$.}\label{evol12}
\end{figure}
%%%%%%%%%%%%%%%%%%%%%%%%%%%%%%%%%%%%%%%%%%%%%%%%%%%%%%%%%%%%

\subsection*{Comparing Recovered population with and without control:}
In Figure 5, we compare the temporal dynamics of recovered birds with and without the controls. The control case is presented in Figure 5a, where the control $u_2$ is varied as $\{0.08,0.48, 0.88\}$. Increasing the value of $u_2$ reduces the rate at which the recovered birds attain a plateau. On the other hand, the recovered birds increases in population in the absence of control irrespective of the choice of $\sigma$. 

\begin{figure}[h!]
\centering
\begin{tabular}{cc}
\includegraphics[width=0.46\textwidth]{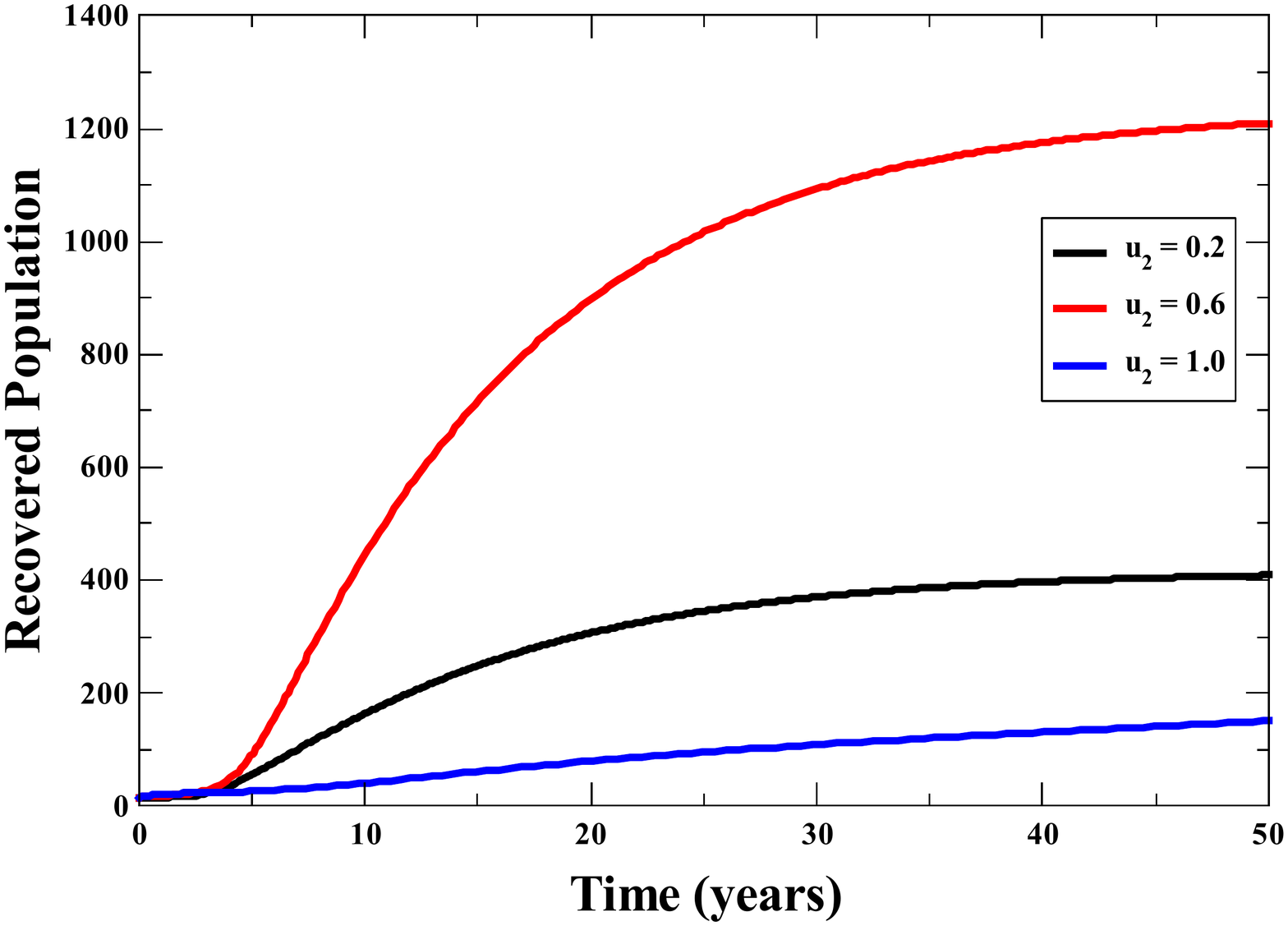}&
\includegraphics[width=0.46\textwidth]{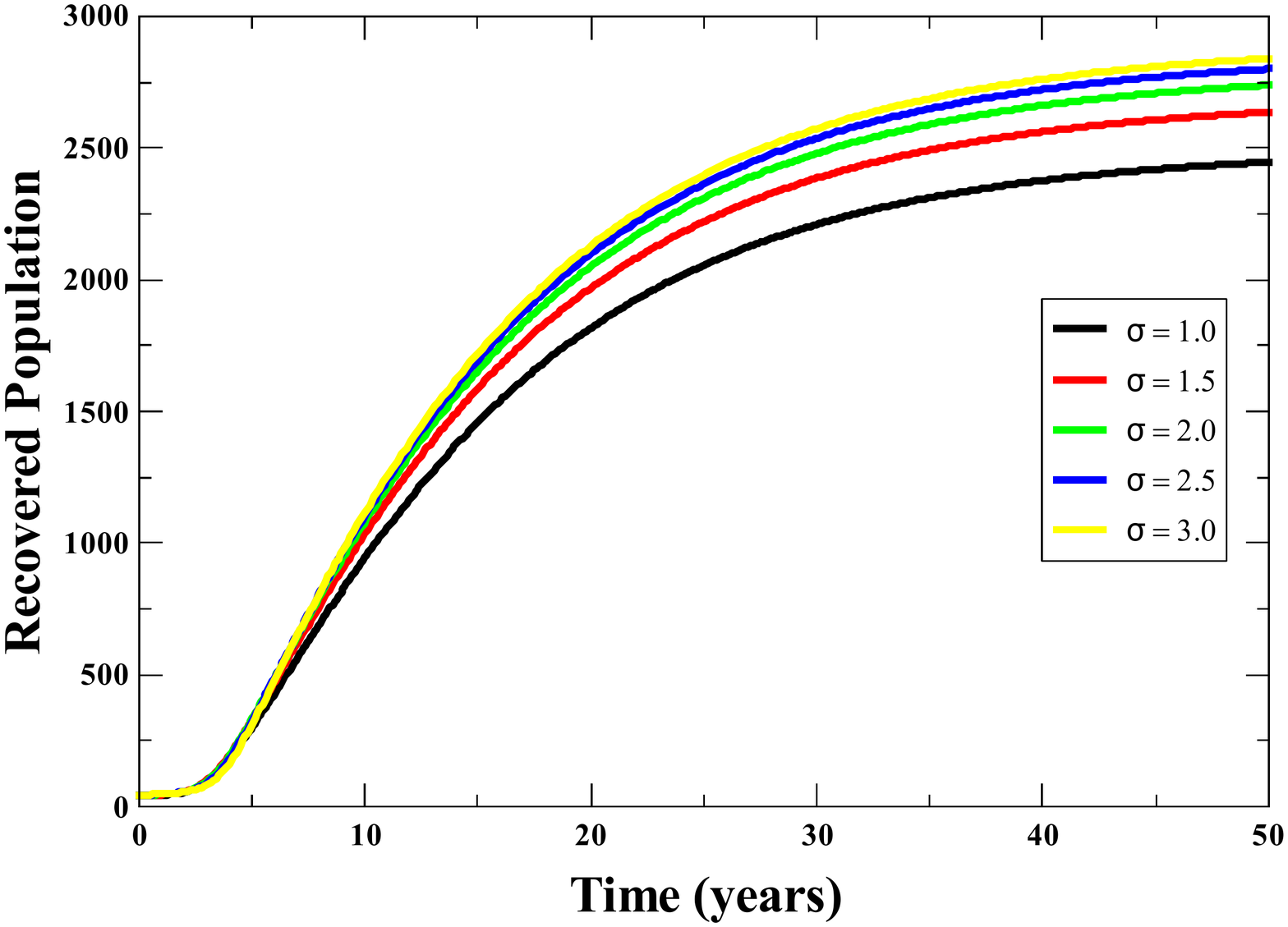}\\
(a) & (b)
\end{tabular}
\caption[Comparing Recovered population with and without control:]{\textbf{Comparing Recovered population with and without control:} Presented in this figure are (a)the temporal dynamics of recovered birds when the control $u_2$ is applied to the model. Here $u_2$ is varied as $\{0.08,0.48, 0.88\}$ (b)the temporal dynamics of recovered birds when there is no control in the model. Here $\sigma$ is variedas $\{1.0,1.5, 2.0,2.5,3.0\}$} \label{evol3}
\end{figure}

%%%%%%%%%%%%%%%%%%%%%%%%%%%%%%%%%%%%%%%%%%%%%%%%%%%%%
%%%%%%%%%%%%%%%%%%%%%%%%%%%%%%%%%%%%%%%%%%%%%%%%%%%%%%%%%%%%%%%
%                       
\section{Discussion}
\section{Conclusion and Recommendation}

\newpage
\newpage
\section{Appendix}

\subsection{Existence of steady states}
The system is in a steady state if, $\frac{d S_B}{dt}=\frac{d E_B}{dt}=\frac{d I_B}{dt}=\frac{d R_B}{dt}=\frac{d S_T}{dt}=\frac{d E_T}{dt}=\frac{d I_T}{dt}= 0$, that is, 

\begin{align}
\tau_BN_B - \beta_1I_TS_B - \beta_2I_BS_B - dS_B &=0 \label{i}\\
\beta_1I_TS_B - \beta_2I_BS_B - \alpha_BE_B - dE_B &=0 \label{ii} \\
\alpha_BE_B  - \sigma I_B - dI_B - \mu I_B &=0 \label{iii} \\
\sigma I_B - dR &=0 \label{iv} \\
\tau_TN_T - \beta_3I_BS_T - \lambda S_TI_T - \delta S_T &=0 \label{v} \\
\beta_3I_BS_T - \theta I_TS_T - \delta E_T - \alpha_TE_T &=0 \label{vi} \\
\alpha_TE_T - \delta I_T &=0 \label{vii}
\end{align}
Solving equations (\ref{i}) to (\ref{vii}) $S^0_B,E^0_B,I^0_T,R^0_B,S^0_T,E^0_T,I^0_T$ for  we have the following:
From (\ref{i}) we have  
$$
\Rightarrow \tau_BN_B = (\beta_1I_T + \beta_2I_B +d)S_B \Rightarrow S^0_B = \frac{\tau_BN^0_B}{(\beta_1I^0_T + \beta_2I^0_B + d)}
$$          
From (\ref{ii})
$$
\beta_1I_TS_B + B_2I_BS_B - \alpha_BE_B -dE_B=0\Rightarrow (\beta_1I^0_T + \beta_2I^0_B)S^0_B = (\alpha_B + d)E^0_B \Rightarrow E^0_B = \frac{\beta_1I^0_T + \beta_2I^0_B)S^0_B}{(\alpha_B + d)}
$$
From (\ref{iii})
$$
\alpha_BE_B + \sigma I_B - dI_B=0\Rightarrow \alpha_BE^0_B =(\sigma + d + \mu)I^0_B = 0 \Rightarrow E^0_B = \frac{(\sigma + d + \mu)I^0_B}{\alpha_B}
$$
From (\ref{iv}),               
$$
\sigma I_B - dR = 0 \Rightarrow \sigma I_B = dR \Rightarrow R^0_B = \frac{\sigma I^0_B}{d}
$$
From (\ref{v}) 
$$
S^0_T = \frac{\tau_TN^0_T}{(\beta_3 I^0_B + \theta I^0_T + \lambda I^0_T + \delta)}
$$
From (\ref{vi})
$$
\beta_3I_BS_T + \theta I_TS_T + \lambda S_TI_T - \delta E_T - \alpha_TE_T=0\Rightarrow (\beta_3I^0_B + \theta I^0_T + \lambda I^0_T)S^0_T = (\delta - \alpha_T)E^0_T 
$$
$$
S^0_T = \frac{\delta - \alpha_T)E^0_T}{(\beta_3I^0_B + \theta I^0_T + \lambda I^0_T)}
$$
From (\ref{vii}),  
$$
\alpha_TE_T - \delta I_T =0 \Rightarrow \alpha_TE_T = \delta I_T \Rightarrow I^0_T = \frac{\alpha_TE^0_T}{\delta}
$$

%%%%%%%%%%%%%%%%%%%%%%%%%%%%%%%%%%%%%%%%%%%%%%%%%%
Equating (b) and (c) we have $E_B^0 = \frac{{({\beta _1}I_T^0 + {\beta _2}I_B^0)S_B^0}}{{({\alpha _B} + d)}} = \frac{{(\sigma  + d + \mu )I_B^0}}{{{\alpha _B}}}\,\,$                 
$$ \Rightarrow ({\beta _1}I_T^0 + {\beta _2}I_B^0)S_B^0{\alpha _B} = ({\alpha _B} + d)(\sigma  + d + \mu )I_B^0$$

$$({\alpha _B}{\beta _1}I_T^0 + {\alpha _B}{\beta _2}I_B^0)S_B^0 = ({\alpha _B} + d)(\sigma  + d + \mu )I_B^0$$

$${\alpha _B}{\beta _1}I_T^0S_B^0 = \left[ {({\alpha _B} + d)(\sigma  + d + \mu ) - {\alpha _B}{\beta _2}} \right]I_B^0$$

$$I_B^0 = \frac{{{\alpha _B}{\beta _1}I_T^0S_B^0}}{{\left[ {({\alpha _B} + d)(\sigma  + d + \mu ) - {\alpha _B}{\beta _2}} \right]}}$$

%%%%%%%%%%%%%%%%%%%%%%%%%%%%%%%%%%%%%%%%%%%%%
%%%%%%%%%%%%%%%%%%%%%%%%%%%%%%%%%%%%%%%%%
Equating (e) and (f) we have 
$S_T^0 = \frac{{{\tau _T}N_T^0}}{\begin{array}{l}
({\beta _3}I_B^0 + \theta I_T^0 + \lambda I_T^0 + \delta )\\

\end{array}}\,\,$= $\frac{{(\delta  - {\alpha _T})E_T^0}}{{({\beta _3}I_B^0 + \theta I_T^0 + \lambda I_T^0)}}\,\,\,\,$

$$\Rightarrow \tau_TN_T^0(\beta_3I_B^0 + \theta I_T^0 + \lambda I_T^0)= ({\beta _3}I_B^0 + \theta I_T^0 + \lambda I_T^0 + \delta )(\delta  - {\alpha _T})E_T^0$$

$$\therefore \,\,\,E_T^0 = \frac{{{\tau _T}N_T^0({\beta _3}I_B^0 + \theta I_T^0 + \lambda I_T^0)}}{{({\beta _3}I_B^0 + \theta I_T^0 + \lambda I_T^0 + \delta )(\delta  - {\alpha _T})}}$$

Therefore the steady (equilibrium) state is:
\begin{equation}
\left. \begin{array}{l}
S_B^0 = \frac{{{\tau _B}N_B^0\,}}{{({\beta _1}I_T^0 + {\beta _2}I_B^0 + d)}}\,\\
\\
\\
E_B^0 = \frac{{(\sigma  + d + \mu )I_B^0}}{{{\alpha _B}}}\,\\
\\
\\
\,I_B^0 = \frac{{{\alpha _B}{\beta _1}I_T^0S_B^0}}{{\left[ {({\alpha _B} + d)(\sigma  + d + \mu ) - {\alpha _B}{\beta _2}} \right]}}\,\,\\
\\
\\
R_B^0 = \frac{{\sigma I_B^0}}{d}\\
\\
\\
S_T^0 = \frac{{{\tau _T}N_T^0}}{{({\beta _3}I_B^0 + \theta I_T^0 + \lambda I_T^0 + \delta )}}\\
\\
\\
E_T^0 = \frac{{{\tau _T}N_T^0({\beta _3}I_B^0 + \theta I_T^0 + \lambda I_T^0)}}{{({\beta _3}I_B^0 + \theta I_T^0 + \lambda I_T^0 + \delta )(\delta  - {\alpha _T})}}\\
\\
\\
\,I_T^0 = \frac{{{\alpha _T}{E_T}}}{\delta }\,
\end{array} \right\}
\end{equation}

The disease free steady (equilibrium) state for the disease is${E_0} = (\frac{{{\tau _B}\mathop N\nolimits_B }}{d},0,0,0,\frac{{{\tau _T}\mathop N\nolimits_T }}{\delta },0,0)$


\begin{thebibliography}{99}
\bibitem{Avian}
Avian Spirochetosis: www.aun.edu.eg|...|ch1-15.htm

\bibitem{Ataliba}
Ataliba,  A., Resende, J. S, Yoshinari N., Labruna M. B. (2007). Isolational molecular Characterization of a Brazilian Strain of Borrelina  Anserina: the agent of fowl Spirochetosis. Res. Vet. Science  83:145 -149. 

\bibitem{Borrelin}
Borrelin anserine (2012). $https://em.wikiret.net|Borrelia_anserina$

\bibitem{Fleming}
Fleming W. H.,Rishel R. W., (1975),Deterministic and Stochastic optimal control, springer Ver lag New York/Berlin.     

\bibitem{Fox}
Fox W., (2010).Tick – borne disease – risks and reality, Borreliosis and Associated disease Awareness UK

\bibitem{Gilbert}
Gilbert L.,Norman R., Laurenson M.K., Reid H.W., Hudson P.J. (2001). Disease persistence and Apparent competition in a three-host community: an empirical and analytical study of large-scale, wild populations. Journal of Animal Ecology, 70(6):1053–1061.

\bibitem{Hee-Dae}
Hee-Dae, K., (2005) Application of Optimal Control Theory to Mathematical Model of Biological Systems,  Inha University, Incheon, Korea.

\bibitem{Kaikabo}
Kaikabo, A.A., Mustapha,  A., Yaroro,  I.(2006) Avian Spirochetosis Associated with Colisepticemia to Free ranging while breasted guinea fowls ( Numida meleagris pallas) in arid region of Nigeria; Animal Production Research Advances,Vol. 2, No. 3, pp 161 – 163

\bibitem{Kriesel}
Kriesel, A; Meyer, M; Peterson, G.(2009) Mathematical Modeling ofTick-Borne Encephalitis in Humans; Journal of undergraduate research at Minnesota state University Markato, Vol. 9, Article 9

\bibitem{Liuyong}
Liuyong, P., Ruan, S., Liu, S., Zaho, Z., Zhang, X. (2015) Transmission Dynamicsand Optimal Control of measles epidemics; Applied Mathematics and Computation 256, 131-147

\bibitem{Nordin}
Nordin N. A.,Rohanin A.,Rashidah A.(2015), Optimal control of vector – borne disease with Direct  transmission.  Journal Teknologi(Sciences and Engineering),76:13,53 – 60.

\bibitem{Porter}
Porter, R. B. (2011) Mathematical Models of a Tick-Borne Disease in a British Game Bird with  Potential Management Strategies

\bibitem{Rosa}
Rosa, R., Pugliese, A.,Norman, R.,Hudson, P.J. (2003) Thresholds for Disease Persistence in Models for Tick-borne Infections including Non-Viraemic Transmission, Extended feeding and Tick  Aggregation; Journal of Theoretical Biology, 224(3):359–376.

\bibitem{Raquel}
Raquel, L. S, Rafuella, T.C., Charles, R.P., Huarrison, S.A., Carlos M.L. and  Adivaldo, F. H. (2009) Avian Spirochetosis in Chickens following Experimental Transmission of Borrelin Anserine by Argas(Persicargas) Miniatus; Avian Diseases Ardi Vol. 53, No.02, 10

\bibitem{Switkes}
Switkes J., Nannyonga B., Mugisha J.Y.T., Nakakawa J. (2016). A mathematical Model for Crimean – Congo Hemorrhagic Fever: Tick-borne Dynamics with Conferred Host Immunity; Journal of biological dynamics vol.10,pp. 59-70

\bibitem{Merck}
The Merck Veterinary Manual (2014) Overview of Avian Spirochetosis; Merck and co. in Kenilworth New Jersey, USA

\end{thebibliography}
\end{document}